\documentclass[11pt]{article}
\usepackage{amsmath}
\usepackage{amsfonts}
\usepackage{amssymb}
\usepackage{mathrsfs}
\usepackage{amsthm}
\usepackage{palatino}
\usepackage[margin=2.5cm, vmargin={1.5cm}]{geometry}

\usepackage{caption}
\usepackage{dcolumn}

\usepackage{times}
\usepackage{graphicx}
\usepackage{color}

\usepackage{pstricks}
\usepackage{pst-plot}

\renewcommand\footnotemark{}

\newrgbcolor{LemonChiffon}{1.0 0.98 0.8}
\newrgbcolor{magenta}{1.0 0.3 0.4}
\newrgbcolor{mistyrose}{1.0 0.9 0.8}
\newrgbcolor{deeppink}{1.0 0.08 0.58}
\newrgbcolor{lightsalmon}{1.0 0.63 0.48}
\newrgbcolor{lightred}{0.7 0.0 0.0}
\newrgbcolor{lightblue}{0.8 0.8 1.0}
\newrgbcolor{indianred}{0.80  0.36  0.36}
\newrgbcolor{byellow}{0.93 0.76 0.3}
\newrgbcolor{light}{0.98 0.92 0.92}

\begin{document}

\title{A vertex and edge deletion game on graphs}

\date{}

\author{
  Cormac ~O'Sullivan\footnote{{\it Date:} Aug 17, 2018.
  \newline \indent \ \ \
  \newline \indent \ \ \
  \newline \indent \ \ \
Support for this project was provided by a PSC-CUNY Award, jointly funded by The Professional Staff Congress and The City
\newline \indent \ \ \
University of New York.}
  }

\maketitle

\def\s#1#2{\langle \,#1 , #2 \,\rangle}

\def\H{{\mathbf{H}}}
\def\F{{\frak F}}
\def\C{{\mathbb C}}
\def\R{{\mathbb R}}
\def\Z{{\mathbb Z}}
\def\Q{{\mathbb Q}}
\def\N{{\mathbb N}}
\def\G{{\Gamma}}
\def\GH{{\G \backslash \H}}
\def\g{{\gamma}}
\def\L{{\Lambda}}
\def\ee{{\varepsilon}}
\def\K{{\mathcal K}}
\def\Re{\mathrm{Re}}
\def\Im{\mathrm{Im}}
\def\PSL{\mathrm{PSL}}
\def\SL{\mathrm{SL}}
\def\Vol{\operatorname{Vol}}
\def\lqs{\leqslant}
\def\gqs{\geqslant}
\def\sgn{\operatorname{sgn}}
\def\res{\operatornamewithlimits{Res}}
\def\li{\operatorname{Li_2}}
\def\lip{\operatorname{Li}'_2}
\def\pl{\operatorname{Li}}
\def\nb{{\mathcal B}}
\def\cc{{\mathcal C}}
\def\nd{{\mathcal D}}
\def\dd{\displaystyle}

\def\n{\mathrm{g}}
\def\f{\phi}
\def\mex{\mathrm{mex}}

\def\clp{\operatorname{Cl}'_2}
\def\clpp{\operatorname{Cl}''_2}
\def\farey{\mathscr F}

\newcommand{\stira}[2]{{\genfrac{[}{]}{0pt}{}{#1}{#2}}}
\newcommand{\stirb}[2]{{\genfrac{\{}{\}}{0pt}{}{#1}{#2}}}
\newcommand{\norm}[1]{\left\lVert #1 \right\rVert}


\newtheorem{theorem}{Theorem}[section]
\newtheorem{lemma}[theorem]{Lemma}
\newtheorem{prop}[theorem]{Proposition}
\newtheorem{conj}[theorem]{Conjecture}
\newtheorem{cor}[theorem]{Corollary}
\newtheorem{assume}[theorem]{Assumptions}
\newtheorem{adef}[theorem]{Definition}


\newcounter{counrem}
\newtheorem{remark}[counrem]{Remark}

\renewcommand{\labelenumi}{(\roman{enumi})}
\newcommand{\spr}[2]{\sideset{}{_{#2}^{-1}}{\textstyle \prod}({#1})}
\newcommand{\spn}[2]{\sideset{}{_{#2}}{\textstyle \prod}({#1})}

\numberwithin{equation}{section}

\bibliographystyle{alpha}

\begin{abstract}
Starting with a graph, two players take turns in either deleting  an edge or deleting a vertex and all incident edges. The player removing the last vertex wins. We review the known results for this game and extend the computation of nim-values to new families of graphs. A conjecture of Khandhawit and  Ye on the nim-values of graphs with one odd cycle is proved. We also see that, for wheels and their subgraphs, this game exhibits a surprising amount of unexplained regularity.
\end{abstract}

\section{Introduction}
Let $G=(V(G),E(G))$ be a finite graph with  vertices $V(G)$ and edges $E(G)$. We allow loops and multiple edges. This is the starting position for the game of {\em graph take-away} (or graph chomp) and its rules are as follows.  Two players take turns in either deleting an edge or deleting a vertex and all incident edges. The player removing the last vertex wins. This impartial game has been studied in \cite{FS,DT,CT,DR,Ri} and most recently \cite{Y}. It is a special case of the general games on partially ordered sets introduced by Gale and Neyman in \cite{GN}; see for example the introductions of \cite{B,FR} for more of their history. The questions in  \cite{GN} were recently answered negatively in \cite{BC}.

For each graph  $G$ we would like to know whether the player going first or the player going second has a winning strategy. According to the Sprague-Grundy theory \cite[Chap. 11]{onag}, each of these graph games has a   nim-value $\n(G)$, and the second player has a winning strategy exactly when $\n(G)=0$. The nim-value of a disjoint union $H_1 \cup H_2$ of two graphs  may be easily calculated from the individual nim-values of $H_1$ and $H_2$ using nim-addition as reviewed in Section \ref{basic}.

\SpecialCoor
\psset{griddots=5,subgriddiv=0,gridlabels=0pt}
\psset{xunit=0.7cm, yunit=0.7cm, runit=0.7cm}
\psset{linewidth=1pt}
\psset{dotsize=5pt 0,dotstyle=*}

\begin{figure}[ht]
\centering
\begin{pspicture}(0,-4)(19.7,1.5) 

\psset{arrowscale=2,arrowinset=0.5}

\psline(0,0)(1,0)(1,1)(0,0)
\psline(8,0)(7,0)(8,1)
\psline(9,0)(10,0)
\psdots(0,0)(1,0)(1,1)(0,0)(8,0)(7,0)(8,1)(9,0)(10,0)
\rput(0.5,-0.7){$G$}

\psline[linecolor=orange]{->}(2.1,0.5)(5.5,0.5)
\rput(3.8,-0.3){First player}
\rput(6.6,0.5){$\Bigg\{$}
\rput(10.4,0.5){$\Bigg\}$}
\rput(8.5,0){\LARGE ,}

\psline[linecolor=orange]{->}(9.5,-0.5)(9.5,-1)(12.5,-1)
\rput(13.6,-1){$\Bigg\{$}
\rput(17.4,-1){$\Bigg\}$}
\psdots(14.5,-1.5)(16,-1.5)(17,-1.5)
\rput(15.5,-1.5){\LARGE ,}

\rput(10.5,-2){Second player}

\psline[linecolor=orange]{->}(7.5,-0.5)(7.5,-3)(12.5,-3)
\rput(13.6,-3){$\Bigg\{$}
\rput(19.4,-3){$\Bigg\}$}
\psline(14,-3.5)(15,-3.5)
\psline(16,-3.5)(17,-3.5)
\rput(15.5,-3.5){\LARGE ,}
\rput(17.5,-3.5){\LARGE ,}
\psdots(14,-3.5)(15,-3.5)(15,-2.5)(16,-3.5)(17,-3.5)(18.5,-3.5)(18.5,-2.5)

\end{pspicture}
\caption{A game of graph take-away on a triangle}
\label{game-tri}
\end{figure}
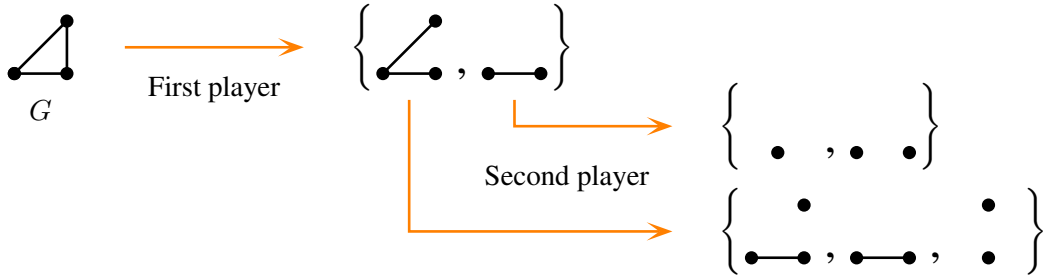

Figure \ref{game-tri} shows the example of a simple game where the starting graph $G$ is a triangle. The first player has six possible moves, giving the two non-isomorphic options listed. The second player's possible replies are listed on the right. The last option in each list, two isolated vertices, is a winning move for player two since they are assured of taking the last vertex in this case. Hence, the player going second has a winning strategy and the nim-value of the triangle $G$ is $0$.

The basic notation and definitions of graph theory we use in this paper are contained in \cite{BM-graph}, for example.
If $G$ is bipartite (two-colorable) then Fraenkel and Scheinerman showed in \cite{FS} that the winning strategy is to restore the number of vertices and the number of edges to even parity.
Using their notation, for any integer $k$ let $k_{(m)}:= k \bmod m$ with  $k_{(m)} \in \{0,1,\dots,m-1\}$. Writing $|V(G)|$ and $|E(G)|$ for the number of vertices and edges of $G$, we define a parity function
$$
\f(G):= |V(G)|_{(2)}+2\bigl( |E(G)|_{(2)}\bigr)
$$
so that $\f(G)$ is $0,$ $1,$ $2$ or $3$.

\begin{prop}{\rm \bf \cite[Cor. 2.2]{FS}} \label{par}
If $G$ is bipartite then $\n(G) = \f(G)$.
\end{prop}

This result was also proved in \cite[Thm. 3]{Y}, with the special case of forests proved in \cite[Thm. 2]{DR}. A more complicated version of Proposition \ref{par} also appeared in \cite[Chap. 5]{Ri}. It follows easily that a tree $T$ has nim-value $|E(T)|_{(2)}+1$ and a cycle graph $C_n$ of length $n \gqs 2$  has nim-value $0$.
A one-cycle (i.e. a vertex with a loop attached) has nim-value $2$. Because of this difference, we must treat loops and longer cycles differently. Consequently, cycles  in this paper refer to cycles of length at least $2$.

\SpecialCoor
\psset{griddots=5,subgriddiv=0,gridlabels=0pt}
\psset{xunit=0.7cm, yunit=0.7cm, runit=0.7cm}
\psset{linewidth=1pt}
\psset{dotsize=5pt 0,dotstyle=*}

\begin{figure}[ht]
\centering
\begin{pspicture}(0,0)(17.5,4) 

\psset{arrowscale=2,arrowinset=0.5}

\psellipse(2,2)(2,1)
\psellipse(7,3)(1.5,0.75)
\psellipse(7,1)(1.5,0.75)
\psellipse(15.5,2)(2,1)

\psdots(4,2)(5.5,3)(5.5,1)(17.5,2)

\psline(4,2)(5.5,3)
\psline(4,2)(5.5,1)
\psline{->}(10,2)(12,2)

\rput(2,2){$G'$}
\rput(7,3){$H_1$}
\rput(7,1){$H_2$}
\rput(15.5,2){$G'$}

\rput(4.2,2.5){$s$}
\rput(5.4,3.5){$v_1$}
\rput(5.4,1.5){$v_2$}
\rput(4.2,0.5){$G$}

\rput(17.8,2.4){$s$}

\end{pspicture}
\caption{Cancellation when $H_1 \cong H_2$}
\label{cancel}
\end{figure}
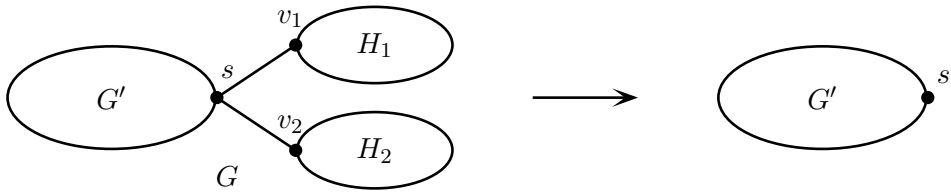

A situation where the  nim-value of a graph may be obtained from a simpler graph is described as follows.  Let $H_1$ and $H_2$ be isomorphic graphs containing corresponding vertices $v_1$ and $v_2$ respectively. Let $G'$ be another graph containing a vertex $s$. Build $G$ from the disjoint graphs $G',$ $H_1$ and $H_2$ by adding the edges $sv_1$ and $sv_2$ as shown in Figure \ref{cancel}. In this situation we say that $G$ has {\em cancellation at $s$} and may be replaced by $G'$ since, as we see in Section \ref{basic}, $\n(G)=\n(G')$.
A graph is {\em reduced} if no cancellation is possible.

Graphs that are not bipartite must contain a loop or an odd cycle. It is reasonable to expect that some non-bipartite graphs $G$ will also have $\n(G)=\f(G)$ if a strategy of eliminating odd cycles can be used.
To determine the nim-values of  graphs with exactly one odd cycle we need to introduce the next definition.

\SpecialCoor
\psset{griddots=5,subgriddiv=0,gridlabels=0pt}
\psset{xunit=0.7cm, yunit=0.7cm, runit=0.7cm}
\psset{linewidth=1pt}
\psset{dotsize=5pt 0,dotstyle=*}
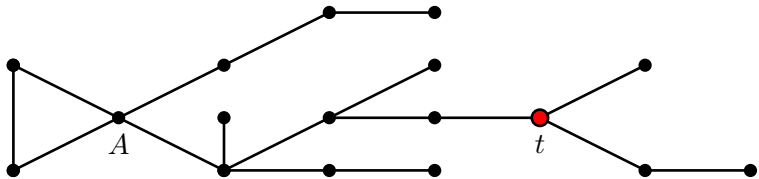
\begin{figure}[ht]
\centering
\begin{pspicture}(4,-0.5)(18,3.5) 

\psset{arrowscale=2,arrowinset=0.5}

\psline(4,2)(4,0)(6,1)(8,2)(10,3)(12,3)
\psline(4,2)(8,0)(10,1)(12,1)(14,1)(16,0)(18,0)
\psline(8,1)(8,0)(10,0)(12,0)
\psline(10,1)(12,2)
\psline(14,1)(16,2)

\psdots(4,2)(4,0)(6,1)(8,2)(10,3)(12,3)(4,0)
\psdots(4,2)(8,0)(10,1)(12,1)(14,1)(16,0)(18,0)(8,1)(8,0)(10,0)(12,0)(12,2)(16,2)

\pscircle[fillstyle=solid,fillcolor=red](14,1){0.18}

\rput(6,0.5){$A$}
\rput(14,0.5){$t$}

\end{pspicture}
\caption{A telescoping vertex $t$}
\label{teleeg}
\end{figure}

\begin{adef}{\rm
Suppose a tree $T$ is attached to  an odd cycle  at  vertex $A$. A vertex of $T$ is {\em telescoping} if, when it is deleted and the resulting graph reduced, all that remains of $T$ that is still connected to the cycle is $A$. }
\end{adef}

Note that when we say that a graph $G_1$ is {\em attached at $v$} to the graph $G_2$, we mean that $G_1\cap G_2$ is the vertex $v$.
In the example in Figure \ref{teleeg},  the tree attached to the $3$-cycle at $A$ has one telescoping vertex as indicated. The properties of telescoping vertices do not seem to have appeared before in the literature, that the author is aware of, and we will see that their study   becomes quite intricate. The main part of this paper, in Sections \ref{st}, \ref{unique} and \ref{1cyc}, establishes the next result.

\begin{theorem} \label{AB2}
Let $G$ be a reduced graph consisting of an odd cycle attached to a tree at one vertex.  Then
\begin{equation*}
  \n(G)= \left\{
           \begin{array}{ll}
            0 & \hbox{if $G$ is just a cycle;} \\
             4 \text{ or more} & \hbox{if there is a telescoping vertex of odd degree;} \\
             \f(G) & \hbox{otherwise.}
           \end{array}
         \right.
\end{equation*}
\end{theorem}

Theorem \ref{AB2} allows us to characterize   when a graph $G$ with exactly one odd cycle  has $\n(G)=\f(G)$.  It also leads to the following result which was conjectured by Khandhawit and  Ye  in  \cite[Conjecture 2]{Y}.

\begin{theorem}{\rm \bf} \label{conj}
Let $G$ be a reduced and possibly disconnected graph containing  one odd cycle and no other cycles or loops. Suppose that two or more vertices of the cycle have degree greater than $2$. Then $\n(G)=\f(G)$.
\end{theorem}

Finding the winner on a graph with one odd cycle is a simple consequence of Theorems \ref{AB2} and \ref{conj} and described in the next corollary. The authors of \cite{Ri}, \cite{Y} and \cite{KN} highlighted this problem.

\begin{cor}
Let $G$ be a reduced and possibly disconnected graph containing  one odd cycle $C$ and no other cycles or loops. Then $\n(G)=0$  if and only if one of these three conditions holds:
\begin{enumerate}
  \item all vertices of $C$ have degree $2$ and $\f(G)=3$;
  \item exactly one vertex of $C$ has degree $>2$, $\f(G)=0$ and $G$ has no telescoping vertices of odd degree;
  \item two or more vertices of $C$ have degree $>2$ and $\f(G)=0$.
\end{enumerate}
\end{cor}

Most of the results in this paper build on those of Khandhawit and  Ye  in  \cite{Y}, such as for graphs with one odd cycle as already mentioned.
In Sections \ref{unbd} and \ref{many} we find the nim-values of further families of graphs, some containing many odd cycles. The following  two propositions show the importance of parity considerations and generalize results in \cite[Appendix B]{Y}.

\SpecialCoor
\psset{griddots=5,subgriddiv=0,gridlabels=0pt}
\psset{xunit=0.5cm, yunit=0.5cm, runit=0.5cm}
\psset{linewidth=1pt}
\psset{dotsize=5pt 0,dotstyle=*}

\begin{figure}[ht]
\centering
\begin{pspicture}(0,0.5)(32,9) 

\psset{arrowscale=2,arrowinset=0.5}

\savedata{\mydata}[
{{8., 5.}, {7.65637, 6.39417}, {6.70419, 7.46895}, {5.36161, 7.97813},
{3.93619, 7.80505}, {2.75447, 6.98937}, {2.08717, 5.71795}, {2.08717,
4.28205}, {2.75447, 3.01063}, {3.93619, 2.19495}, {5.36161, 2.02187},
{6.70419, 2.53105}, {7.65637, 3.60583},{8., 5.}}
]
\dataplot[plotstyle=line]{\mydata}
\dataplot[plotstyle=dots]{\mydata}

\savedata{\mydatab}[
{{8., 5.}, {7.60313, 6.3516}, {6.53854, 7.27408}, {5.14421,
  7.47455}, {3.86285, 6.88937}, {3.10127, 5.70433}, {3.10127,
  4.29567}, {3.86285, 3.11063}, {5.14421, 2.52545}, {6.53854,
  2.72592}, {7.60313, 3.6484}, {8., 5.}}
]
\dataplot[plotstyle=line]{\mydatab}
\dataplot[plotstyle=dots]{\mydatab}

\savedata{\mydatac}[
{{8., 5.}, {7.53209, 6.28558}, {6.3473, 6.96962}, {5.,
  6.73205}, {4.12061, 5.68404}, {4.12061, 4.31596}, {5.,
  3.26795}, {6.3473, 3.03038}, {7.53209, 3.71442}, {8., 5.}}
]
\dataplot[plotstyle=line]{\mydatac}
\dataplot[plotstyle=dots]{\mydatac}

\savedata{\mydataww}[
{{27., 2.}, {27,5}, {29.3455, 3.12953}, {29.9248, 5.66756}, {27,5},  {28.3017,
  7.70291}, {25.6983, 7.70291},  {27,5}, {24.0752, 5.66756}, {24.6545,
  3.12953},  {27,5}, {27., 2.}}
]
\dataplot[plotstyle=line]{\mydataww}

\savedata{\mydataw}[
{{27., 2.}, {29.3455, 3.12953}, {29.9248, 5.66756}, {28.3017,
  7.70291}, {25.6983, 7.70291}, {24.0752, 5.66756}, {24.6545,
  3.12953}, {27., 2.}}
]
\dataplot[plotstyle=line]{\mydataw}
\dataplot[plotstyle=dots]{\mydataw}
\psdot(27,5)


\savedata{\mydataz}[
{{20., 5.00005}, {19.0163, 6.32177}, {17.6191, 7.19499}, {16.,
  7.5}, {14.3809, 7.19499}, {12.9837, 6.32177}, {12., 5.00005}}
]
\dataplot[plotstyle=line]{\mydataz}
\dataplot[plotstyle=dots]{\mydataz}

\savedata{\mydataza}[
{{19.9997, 5.00017}, {18.0614, 5.74626}, {16., 6.}, {13.9386,
  5.74626}, {12.0003, 5.00017}}
]
\dataplot[plotstyle=line]{\mydataza}
\dataplot[plotstyle=dots]{\mydataza}

\savedata{\mydatazb}[
{{19.9997, 4.99983}, {18.0614, 4.25374}, {16., 4.}, {13.9386,
  4.25374}, {12.0003, 4.99983}}
]
\dataplot[plotstyle=line]{\mydatazb}
\dataplot[plotstyle=dots]{\mydatazb}

\savedata{\mydatazc}[
{{20., 4.99995}, {19.0163, 3.67823}, {17.6191, 2.80501}, {16.,
  2.5}, {14.3809, 2.80501}, {12.9837, 3.67823}, {12., 4.99995}}
]
\dataplot[plotstyle=line]{\mydatazc}
\dataplot[plotstyle=dots]{\mydatazc}

\rput(16,3.2){$z_4$}
\rput(16,4.7){$z_3$}
\rput(16,6.7){$z_2$}
\rput(16,8.2){$z_1$}

\rput(5,0.9){$r=3$ odd cycles}
\rput(16,0.9){$k=4$ paths between $P$ and $Q$}
\rput(27,0.9){$W_n$ for $n=7$}
\rput(11.2,5){$P$}
\rput(20.8,5){$Q$}

\end{pspicture}
\caption{Some graph families}
\label{families}
\end{figure}
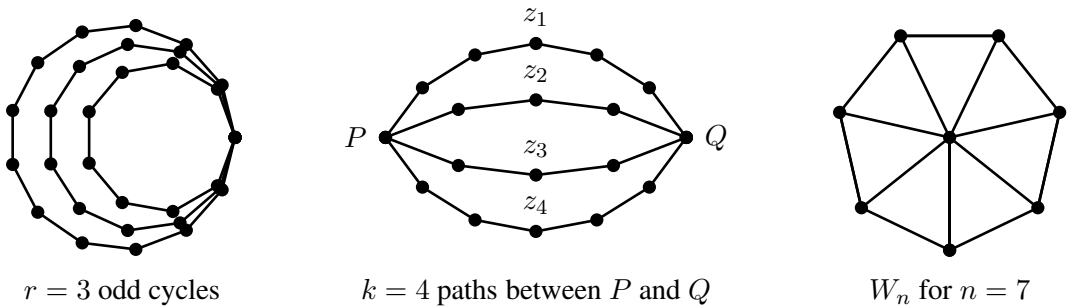

\begin{prop} \label{oddcyc}
If $r$ odd cycles are attached at one vertex, as shown for example on the left of Figure \ref{families}, then the nim-value of this graph is $0$ if $r$ is odd and $1$ if $r$ is even.
\end{prop}

\begin{prop} \label{paths}
Let $G$ be a graph consisting of $k$ paths of positive lengths $z_1, \dots, z_k$ linking vertices $P$ and $Q$, as in the example in the middle of Figure \ref{families}. Set $Z:=\sum_{i=1}^k z_k$. Then
\begin{equation*}
  \n(G)=\left\{
          \begin{array}{ll}
            0 & \hbox{if $k$ is even;} \\
            1 & \hbox{if $k$ is odd and $Z$ is even;} \\
            2 & \hbox{if $k$ is odd and $Z$ is odd.}
          \end{array}
        \right.
\end{equation*}
\end{prop}

For $n\gqs 3$ the  {\em wheel graph} $W_n$
is constructed by joining a central hub vertex to each vertex of the cycle graph $C_n$. These joining edges are called spokes. See Figure \ref{families} for $W_{7}$. In \cite{Y} they make
an elegant conjecture about the nim-values of wheel graphs.

\begin{conj}{\rm \bf \cite[Conjecture 3]{Y}} \label{conjw}
We have $\n(W_n)=1$ for all $n \gqs 3$.
\end{conj}

 They show with a symmetry argument that this conjecture is true for all $n$ even and by a computation that it is true for $n=3,5,7$. By combining  symmetry arguments with a computer search we extend this range and prove in Theorem \ref{wheel} that  Conjecture \ref{conjw} is true for $3\lqs n\lqs 25$.
We make further conjectures about the nim-values of subgraphs of $W_n$ in Section \ref{wh}.

We close this introduction by noting that there are other ways to play impartial games (i.e. with rules the same for both players) on  undirected graphs. Games that appeared before graph take-away are {\em node kayles} and {\em arc kayles} which were introduced in the study of computational complexity in \cite{SCH}. The moves in node kayles consist of  removing any vertex along with all its neighboring vertices. The moves in arc kayles involve choosing an edge and removing its endpoint vertices (and all incident edges). Three recent games on graphs that are also similar to graph take-away, though they perhaps have less  structure, are the following. In the {\em odd/odd vertex deletion game} players may only remove vertices of odd degree; see  \cite{NO,odd14}.  With {\em graph nim}, as in for example \cite{LC16}, a player on their turn removes any positive number of edges incident to a single vertex.  {\em Grim} is introduced in \cite{grim16} and a player removes a vertex, all incident edges and any vertices that have become isolated.  For all these games, as usual, the first player unable to play loses. Arc kayles and grim are examples of the octal games on graphs studied in \cite{OCT}.

\vskip .15in
\textbf{Acknowledgements.}
The author thanks Andries E. Brouwer  for providing a copy of \cite{DR} and Tirasan Khandhawit for the reference \cite{Ri}.

\section{Basic methods} \label{basic}
We recall more of the theory of impartial games from, for example, \cite[Chap. 11]{onag}, \cite[Chap. 3]{BCG}. The nim-value of a graph game may be calculated inductively as follows. The empty graph has value $0$ and if a graph $G$ has the subgraph options (moves) $G_1, G_2, \dots, G_m$ for the first player then
\begin{equation}\label{mex}
\n(G)=\mex\bigl(\{\n(G_1),\n(G_2), \dots, \n(G_m)\}\bigr)
\end{equation}
where $\mex$ denotes the minimal non-negative integer excluded from the set.

 If $G$ is  disconnected and equal to a disjoint union of $n$  subgraphs $H_1, \dots, H_n$ then
\begin{equation}\label{oplus}
  \n(G)=\n(H_1)\oplus \cdots \oplus \n(H_n)
\end{equation}
where $\oplus$ is the xor operation (binary addition without carry) and called nim-addition in this context.
The general relation \eqref{oplus} follows from the $n=2$ case of \eqref{oplus} which is straightforward to prove with \eqref{mex}.  Let $N(r)$ be the set of nim-values of the possible moves of $H_r$  with $m_r:=\mex(N(r))$. Set $M:=m_1\oplus \cdots \oplus m_n$ and note that $M\oplus m_r$ is the same nim-sum with $m_r$ removed. The {\em addition of  games relation} we will need later is
\begin{equation}\label{addg}
  M=\mex\left(\bigl\{  M \oplus m_r \oplus a  \, \big|\, 1\lqs r\lqs n, \ a \in N(r)\bigr\}\right)
\end{equation}
and it is equivalent to \eqref{oplus}.

In the case that $G$ is a disjoint union of bipartite  subgraphs $H_1, \dots, H_n$, then $G$ is also bipartite and \eqref{oplus} becomes
\begin{equation}\label{fff}
  \f(G)=\f(H_1)\oplus \cdots \oplus \f(H_n),
\end{equation}
which is easy to verify directly.

\SpecialCoor
\psset{griddots=5,subgriddiv=0,gridlabels=0pt}
\psset{xunit=0.7cm, yunit=0.7cm, runit=0.7cm}
\psset{linewidth=1pt}
\psset{dotsize=5pt 0,dotstyle=*}

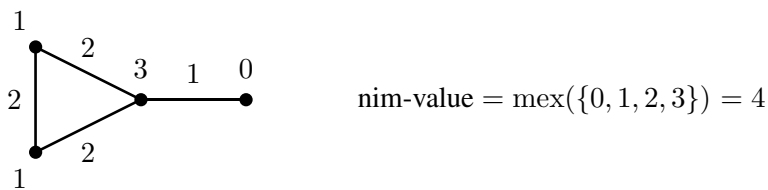
\begin{figure}[ht]
\centering
\begin{pspicture}(-4,0.4)(15,4) 

\psset{arrowscale=2,arrowinset=0.5}

\psdots(0,1)(0,3)(2,2)(4,2)

\psline(0,1)(0,3)(2,2)(4,2)(2,2)(0,1)

\rput(-0.3,3.5){$1$}
\rput(-0.3,0.5){$1$}
\rput(-0.4,2){$2$}
\rput(2,2.6){$3$}
\rput(4,2.6){$0$}
\rput(3,2.5){$1$}
\rput(1,3){$2$}
\rput(1,1){$2$}

\rput(10,2){nim-value $= \mex(\{0,1,2,3\})=4$}

\end{pspicture}
\caption{A graph with nim-value $4$}
\label{simex}
\end{figure}

For a non-bipartite example, we compute the nim-value of a triangle with an edge attached as shown in Figure \ref{simex}. The nim-values of the possible moves are indicated. Removing the degree $3$ vertex, for instance, leaves two trees and the resulting graph has nim-value $1\oplus 2=3$. Moves with values $0,$ $1$ and $2$ are also possible. Hence the nim-value of this graph is $4$.

The symmetry argument we mentioned in the introduction is contained in the next lemma. It may be used to replace a graph game with a smaller one that has the same nim-value.

\begin{lemma}{\rm (The symmetry lemma.)} \label{twin}
Let $G=(V(G),E(G))$ be a graph and $\tau:G\to G$ an automorphism with the following properties:
\begin{enumerate}
  \item $\tau^2$ is the identity,
  \item for all $v\in V(G)$, the vertices $v$ and $\tau(v)$ are not connected by an edge in $G$.
\end{enumerate}
Let $G^\tau$ be the subgraph of $G$ on which $\tau$ acts as the identity. Then $\n(G)=\n(G^\tau)$.
\end{lemma}
\begin{proof}
Let $H$ be a copy of $G^\tau$ and consider the game played on the disjoint union of $G$ and $H$. The second player has the winning strategy of responding to any move in $H$ with the same move in $G^\tau$ and vice versa. Any removal of vertices or edges outside of $G^\tau$ is answered by removing their image under $\tau$. Hence $0=\n(G\cup H)=\n(G)\oplus\n(H)$ and the result follows.
\end{proof}

\SpecialCoor
\psset{griddots=5,subgriddiv=0,gridlabels=0pt}
\psset{xunit=0.7cm, yunit=0.7cm, runit=0.7cm}
\psset{linewidth=1pt}
\psset{dotsize=5pt 0,dotstyle=*}

\begin{figure}[ht]
\centering
\begin{pspicture}(0,-0.5)(18,2.2) 

\psset{arrowscale=2,arrowinset=0.5}

\psline[linecolor=orange](0,1)(0,2)(1,1)(1,2)(0,2)
\psline[linecolor=orange](0,1)(0,0)(1,1)(1,0)(0,0)
\psline(0,1)(4,1)
\psline(1,1)(2,2)(2,0)(4,0)
\psline(2,2)(3,1)
\psline(2,1)(3,0)

\psdots[linecolor=orange](0,0)(0,2)(1,0)(1,2)
\psdots(0,1)(1,1)(2,0)(2,1)(2,2)(3,0)(3,1)(4,0)(4,1)

\psline[linecolor=orange](8,1)(12,1)
\psline[linecolor=orange](9,1)(10,2)(11,1)
\psline(10,2)(10,0)(12,0)
\psline(10,1)(11,0)

\psdots[linecolor=orange](8,1)(9,1)(11,1)(12,1)
\psdots(10,1)(10,0)(10,2)(11,0)(12,0)

\psline(16,2)(16,0)(18,0)
\psline(16,1)(17,0)
\psdots(16,0)(17,0)(18,0)(16,1)(16,2)

\psline{->}(5.5,1)(6.5,1)
\psline{->}(13.5,1)(14.5,1)

\end{pspicture}
\caption{Using the symmetry lemma to simplify}
\label{syme}
\end{figure}
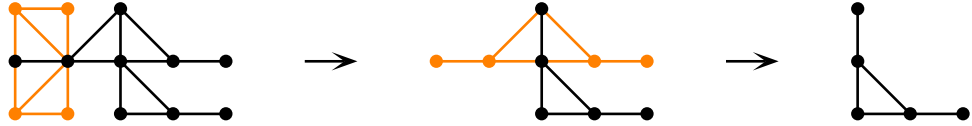

Lemma \ref{twin} and its proof are  based on a combination of \cite[Lemma 3]{FS} and \cite[Prop. 3]{DR}. This important principle of simplifying a game position using symmetry is also used in \cite[Lemma 3]{FS}, \cite[Thm. 6]{Ri} for hypergraphs, \cite[Thm. 1]{Ri}, \cite[Thm. 1]{Y} for simplicial complexes, and \cite[Lemma 2.21]{FR}, \cite[Sect. 2.4]{BC} for posets.

We may list here some applications of the symmetry lemma:
\begin{itemize}
  \item An easy application   shows that a graph with two edges connecting a pair of vertices has the same nim-value when both edges are removed. So $m$ edges between a pair of vertices (or $m$ loops at a single vertex) simplify to a single edge if $m$ is odd and no edge of $m$ is even.
  \item  A second application of Lemma \ref{twin} leads to
\begin{equation}\label{kn}
  \n(K_n)=n_{(3)}
\end{equation}
where $K_n$ is the complete graph on $n$ vertices.
Removing a vertex of $K_n$ gives $K_{n-1}$ and removing an edge  leaves $K_{n-2}$ since we may take a $\tau$ in the symmetry lemma  that switches the deleted edge's endpoints and fixes the remaining vertices. Then \eqref{kn} follows using \eqref{mex} and induction. A similar argument for the complete multipartite graph $K_{n_1,n_2, \dots, n_r}$ shows
\begin{equation}\label{kn2}
  \n(K_{n_1,n_2, \dots, n_r})=\left((n_1)_{(2)}+(n_2)_{(2)}+ \dots + (n_r)_{(2)}\right)_{(3)}.
\end{equation}
Formulas \eqref{kn} and \eqref{kn2} first appeared in \cite{FS}. (In \cite{GN} they showed that $\n(K_n)=0$ if and only if $3 \mid n$.) We generalize \eqref{kn} in Theorem \ref{kngen} by adding loops to $K_n$.
  \item In the preprint \cite{KN}, an argument based on Lemma \ref{twin} succeeds in computing the nim-values of generalized Kneser graphs. For example,  the Petersen graph is shown to have nim-value $2$.
  \item Clearly the cancellation  described in the introduction and pictured in Figure \ref{cancel} is a special case of the symmetry lemma.
\end{itemize}

\begin{adef}{\rm
Recall that a graph is {\em reduced} if no further cancellations are possible. A graph is {\em simplified} if no further non-trivial applications of the symmetry lemma are possible. }
\end{adef}

As we have seen,  simplified implies  reduced.
The next lemma shows that the reduced version of a graph is well-defined up to isomorphism.

\begin{lemma} \label{red}
Let $G$ be a graph. Suppose $G_1$ and $G_2$ are graphs obtained by reducing $G$. Then $G_1$ and $G_2$ are isomorphic.
\end{lemma}
\begin{proof}
We use induction on the number of vertices of $G$. The lemma is true in the base case of an empty graph. If $G$ does not reduce then $G_1=G_2=G$. Otherwise, suppose $G$ has a vertex $s$ where cancellation is possible. Let there be a total of $k$ isomorphic copies of $H$ attached to $s$. Cancelling in pairs we obtain $G'$ with all of the $H$s deleted if $k$ is even and  one copy left if $k$ is odd. Now $G_1$ and $G_2$ must be subgraphs of $G'$ if $k$ is even - or else they are not simplified. If $k$ is odd then we may say that $G_1$ and $G_2$ are isomorphic to subgraphs of $G'$. Replace $G_1$ and $G_2$ by these isomorphic  subgraphs of $G'$ for clarity. By induction, the lemma is true for $G'$ and so $G_1$ and $G_2$ are isomorphic.
\end{proof}

In the case of a tree, we claim that any application of the symmetry lemma  must involve just  cancellation. To see this, suppose a $\tau$ from that lemma maps vertex $a$ to $b$ for $a\neq b$. We use the notation $P(x,y)$  for the unique simple path on a tree connecting two vertices $x$ and $y$. Then $\tau(P(a,b))=P(a,b)$, since $\tau^2$ is the identity, and $P(a,b)$ must have an even number of edges with the middle vertex $r$ fixed by $\tau$. In this way we see that the tree attached at $r$ containing $a$ cancels with the tree attached at $r$ containing $b$. Any further vertices of the tree not fixed by $\tau$ will cancel in the same way. This proves the claim and shows that if a tree is reduced then it is simplified.

\section{Telescoping vertices and cancellation} \label{st}

The graphs we study in Sections \ref{st}, \ref{unique} and \ref{1cyc} contain a single odd cycle and no further cycles or loops. In general, such a graph $G$ consists of a cycle component, made up of a cycle with trees attached to its vertices, and a number of disconnected trees. It is easy to see that any application of the symmetry lemma  to a graph with exactly one odd cycle must act as the identity on this cycle. It follows from this, and the discussion at the end of the last section, that  $G$ above is simplified if and only if it is  reduced.

We next develop some properties of cancellation and telescoping that we will need. For a graph $G$ containing a vertex $v$, let $G-v$ be the graph obtained  by deleting $v$ and all edges incident with $v$. Recall that cancellation occurs at a vertex $s$, say, when we have the situation in Figure \ref{cancel}.

\SpecialCoor
\psset{griddots=5,subgriddiv=0,gridlabels=0pt}
\psset{xunit=0.7cm, yunit=0.7cm, runit=0.7cm}
\psset{linewidth=1pt}
\psset{dotsize=5pt 0,dotstyle=*}
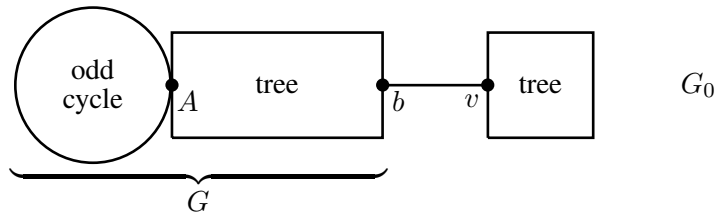
\begin{figure}[ht]
\centering
\begin{pspicture}(-3,-1)(13,3.5) 

\psset{arrowscale=2,arrowinset=0.5}

\pscircle(0.5,1.5){1.5}
\rput(0.5,1.8){odd}
\rput(0.5,1.2){cycle}
\rput(2.3,1.2){$A$}

\psline(2,0.5)(6,0.5)(6,2.5)(2,2.5)(2,0.5)
\psline(6,1.5)(8,1.5)
\psline(8,0.5)(10,0.5)(10,2.5)(8,2.5)(8,0.5)

\psdots(2,1.5)(6,1.5)(8,1.5)

\rput(4,1.5){tree}
\rput(9,1.5){tree}
\rput(6.3,1.2){$b$}
\rput(7.7,1.2){$v$}
\rput(12,1.5){$G_0$}
\rput(2.5,-0.5){$\underbrace{\quad \qquad \qquad \qquad \qquad \qquad \qquad}_{\text{\normalsize $G$}}$}

\end{pspicture}
\caption{Examining possible cancellation in $G$ when $v$ is removed}
\label{te}
\end{figure}

\begin{lemma} \label{can}
Attach  an odd cycle  to a tree at vertex $A$ and, from a vertex $b$ in this tree, join another tree using the edge $bv$ as shown in Figure \ref{te}. Let $G_0$ be this graph and assume it is  reduced.  Let $G$ be the  connected component of  the cycle that remains when $v$ is deleted. If cancellation is now possible in $G$ at a vertex $s$ then:
\begin{enumerate}
  \item The vertex $s$ is  in $P(A,b)-b$.
\item The vertex $s$ is unique.
  \item Cancellation at $s$ can only remove vertices $x$ which have the property that $P(A,x)$ contains  $s$.
  \item Suppose the cancellation at $s$ is carried out. If another cancellation is now possible then it must occur at a unique vertex in  $P(A,s)-s$.
\end{enumerate}
\end{lemma}
\begin{proof}
Suppose the   isomorphic subtrees $H_1$ and $H_2$ cancel in $G$ at $s$,  as seen in Figure \ref{cancel}, with $G'$ being the graph that remains after cancellation. Note that $H_1$ and $H_2$ must each contain at least one vertex. We have $A\in G'$ since the cancellation at $s$ corresponds to an automorphism from the symmetry lemma and $A$ is fixed by any such automorphism. Part (iii) follows from this observation. We must have $v$ adjacent to one of the vertices of $H_1$ or $H_2$ in $G_0$ since $G_0$ is reduced. Hence $b$ is in $H_1$ or $H_2$ and  this proves (i).

Without losing generality,  assume that $b\in H_1$.
Suppose $G$ has cancellation at the vertex $s'$ as well as $s$. Then  $s' \in P(A,b)-b$ by part (i). 
We claim that $s'$ cannot be a vertex of $H_1$. Suppose $s'\in H_1$ and that $w\in H_2$ corresponds to $s'$ under the isomorphism between $H_1$ and $H_2$. Then cancellation at $s'$ means that there is also cancellation at $w$. However this contradicts our requirement from part (i)  that $w \in P(A,b)-b$  and so we have proved our claim. It follows that $s'\in G' \cap P(A,b) = P(A,s)$. Switching the roles of $s$ and $s'$ shows that $s\in P(A,s')$ as well. Consequently we must have $s=s'$, proving  (ii).

Label $G_1$ the subgraph of $G$ obtained by removing $H_2$. Let $v_1 \in H_1$ be adjacent to $s$. Then replace $G_0$ by $G_1$, $v$ by $v_1$, $b$ by $s$ and apply parts (i), (ii) to obtain (iv).
\end{proof}

If an odd cycle has trees attached to a number of its vertices, then Lemma \ref{can} applies to each of these trees separately. Since the odd cycle is fixed by the symmetry lemma, there can be no cancellation between trees attached at different vertices of the cycle.

Let $G$ be a reduced graph consisting of an odd cycle attached to a tree at one point. Suppose $G$ contains a telescoping vertex, so that deleting it gives a cycle component that may be reduced to the cycle in $q$ cancellations. Label this vertex $a_{q+1}$. We next give a precise description of the structure of $G$.

The simplest case of $q=0$ is displayed in Figure \ref{tele}; removing $a_1$ disconnects the tree  $T_1$ and only the cycle is left.
\SpecialCoor
\psset{griddots=5,subgriddiv=0,gridlabels=0pt}
\psset{xunit=0.7cm, yunit=0.7cm, runit=0.7cm}
\psset{linewidth=1pt}
\psset{dotsize=5pt 0,dotstyle=*}
\begin{figure}[ht]
\centering
\begin{pspicture}(-1,-4)(18,3.5) 

\psset{arrowscale=2,arrowinset=0.5}

\pscircle(0.5,1.5){1.5}
\rput(0.5,1.8){odd}
\rput(0.5,1.2){cycle}
\rput(2.2,0.9){$A$}

\psline(4,2)(6,2)(6,3)(4,3)(4,2)
\psline(4,0)(6,0)(6,1)(4,1)(4,0)
\psline(8.5,2)(10.5,2)(10.5,3)(8.5,3)(8.5,2)
\psline(8.5,0)(10.5,0)(10.5,1)(8.5,1)(8.5,0)
\psline(13,2)(15,2)(15,3)(13,3)(13,2)

\psline(4,2.5)(2,1.5)(4,0.5)
\psline(8.5,2.5)(6,2.5)(8.5,0.5)
\psline(10.5,2.5)(13,2.5)
\psdots(2,1.5)(4,2.5)(4,0.5)(6,2.5)(6,0.5)(8.5,2.5)(8.5,0.5)(10.5,2.5)(10.5,0.5)(13,2.5)

\pscircle(6.5,-2.5){1.5}
\psline(10.5,-2)(12.5,-2)(12.5,-3)(10.5,-3)(10.5,-2)
\psline(8,-2.5)(10.5,-2.5)
\psdots(8,-2.5)(10.5,-2.5)

\rput(11.5,-2.5){$T_{q+1}$}
\rput(9.85,-2.1){$a_{q+1}$}

\pscircle[fillstyle=solid,fillcolor=red](13,2.5){0.18}
\pscircle[fillstyle=solid,fillcolor=red](10.5,-2.5){0.18}

\rput(6.5,-2.2){odd}
\rput(6.5,-2.8){cycle}
\rput(8.2,-3){$A$}

\rput(5,2.5){$T_1$}
\rput(3.6,2.9){$a_1$}
\rput(6.4,2.9){$b_1$}
\rput(5,0.5){$U_1$}
\rput(3.6,0.1){$c_1$}
\rput(6.4,0.1){$d_1$}

\rput(9.5,2.5){$T_q$}
\rput(8.1,2.9){$a_q$}
\rput(10.9,2.9){$b_q$}
\rput(9.5,0.5){$U_q$}
\rput(8.1,0.1){$c_q$}
\rput(10.9,0.1){$d_q$}

\rput(14,2.5){$T_{q+1}$}
\rput(12.35,2.9){$a_{q+1}$}

\rput(17,1.5){$q=2$}
\rput(17,-2.5){$q=0$}

\end{pspicture}
\caption{Graphs with telescoping vertices $a_{q+1}$}
\label{tele}
\end{figure}
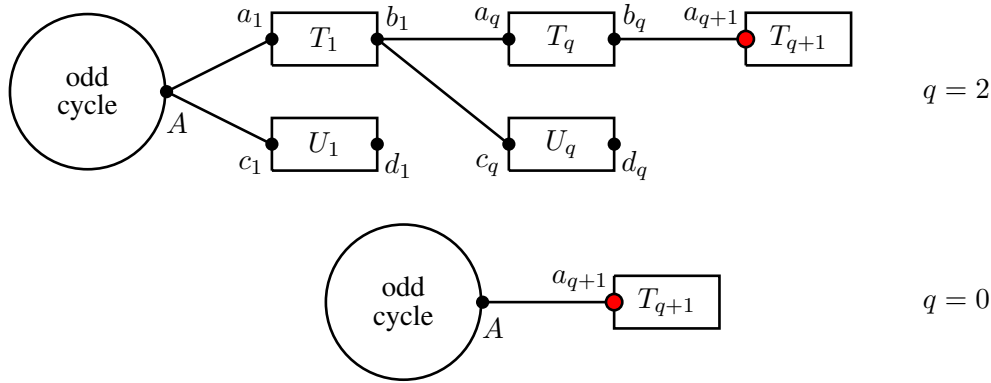
For $q\gqs 1$, Lemma \ref{can} part (ii) implies the first cancellation after deleting $a_{q+1}$ must be at a unique vertex we label $b_{q-1}$ with two isomorphic trees $T_q$ and $U_q$ cancelling as seen in Figure \ref{tele}. If further cancellation is possible then Lemma \ref{can} part (iv) shows it must be at a unique vertex we label $b_{q-2}$. Continuing in this way we obtain the well-defined cancellation vertices $\{b_{q-1}, b_{q-2}, \dots, b_1,A\}$ and in particular the number $q$ is well-defined.

We have shown that $G$ must look like the graphs in Figure \ref{tele}, where $a_{q+1}$ is the telescoping  vertex.
For $1\lqs i\lqs q$ the trees $T_i$ and $U_i$ are isomorphic with $a_i$ corresponding to $c_i$ and  $b_i$ corresponding to $d_i$. Note that these vertices may have large degrees. On the other hand, $T_i$ may be a single vertex in which case $a_i$ and $b_i$ coincide (and then similarly for $U_i$).
In this way we see that having a telescoping vertex can be a fairly complicated situation. It is straightforward, at least, to prove these necessary conditions.

\begin{lemma} \label{teleprops}
Let $G$ be a reduced graph consisting of an odd cycle attached at vertex $A$ to a tree. Let the set of vertices of the tree a distance $x$ from $A$ be labelled $S_x$. Suppose $G$ has a telescoping vertex $v\in S_d$ for  $d\gqs 1$. Then the following are true:
\begin{enumerate}
  \item We have $\deg A \lqs 4$ and $\deg A=3$ if and only if $d=1$.
  \item The numbers $|S_1|, |S_2|, \dots, |S_{d-1}|$ are even and $|S_d|$ is odd.
  \item The total degree of the vertices in $S_d-v$ is even.
\end{enumerate}
\end{lemma}

For  $G$, as in Lemma \ref{teleprops}, it follows from (ii) above that if there are any further telescoping vertices then they must also be in $S_d$. In fact we will show in the next section that telescoping vertices are unique. This requires  some more definitions, notation and a couple of lemmas on subgraphs of isomorphic rooted trees.

Suppose $T$ is a tree with root vertex $r$. We use the notation $T(r)$ for this rooted tree. For any vertex $x$ of $T$ we set
$\rho_T(x)$ to be the subgraph of $T$ induced by the set of vertices $\{v\in T\ :\ x\in P(r,v)\}$.
Then $\rho_T(x)$ consists of $x$ and everything in $T$ on the other side of $x$ from the root.

 An isomorphism of rooted trees $\psi:T(r)\to U(s)$ is a graph isomorphism mapping $T$ to $U$ and $r$ to $s$.  For two vertices $x,y$ in $T$, the distance between them is the length (number of edges) of $P(x,y)$. Clearly
\begin{equation*}
  \psi(P(x,y))=P(\psi(x),\psi(y))
\end{equation*}
and so the isomorphism $\psi$ preserves distance. We also have
\begin{equation*}
  \psi(\rho_T(x))=\rho_{U}(\psi(x)).
\end{equation*}

Let $C$ and $T'$ be  trees with $C$    attached  at  $q$ to $T'$. Recall this means that $C\cap T'=q$.  Let $T=C\cup T'$ and choose a root for $T$. If this root  is in $T'$ then $C \subseteq \rho_T(q)$. There may be  another tree $C'$ attached at $q$ so that $\rho_T(q) = C\cup C'$. Note that for any $v \in C-q$ we have $\rho_T(v) \subseteq C-q$.

\SpecialCoor
\psset{griddots=5,subgriddiv=0,gridlabels=0pt}
\psset{xunit=0.7cm, yunit=0.7cm, runit=0.7cm}
\psset{linewidth=1pt}
\psset{dotsize=5pt 0,dotstyle=*}
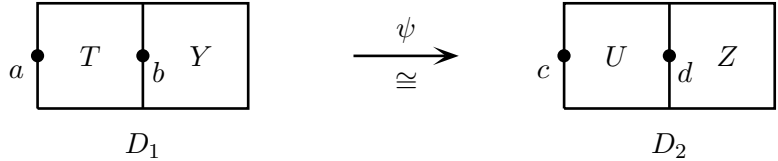
\begin{figure}[ht]
\centering
\begin{pspicture}(0,-1)(14,2.5) 

\psset{arrowscale=2,arrowinset=0.5}

\psline(0,0)(4,0)(4,2)(0,2)(0,0)
\psline(2,0)(2,2)
\psline(10,0)(14,0)(14,2)(10,2)(10,0)
\psline(12,0)(12,2)

\psline{->}(6,1)(8,1)

\psdots(0,1)(2,1)
\psdots(12,1)(10,1)

\rput(7,1.5){$\psi$}
\rput(7,0.5){$\cong$}

\rput(1,1){$T$}
\rput(3.1,1){$Y$}
\rput(2.3,0.7){$b$}
\rput(-0.4,0.7){$a$}

\rput(13.1,1){$Z$}
\rput(12.3,0.7){$d$}
\rput(9.6,0.7){$c$}
\rput(11,1){$U$}

\rput(2,-0.7){$D_1$}
\rput(12,-0.7){$D_2$}

\end{pspicture}
\caption{The isomorphic trees in Lemmas \ref{3pos} and \ref{tyuz}}
\label{tu}
\end{figure}

\begin{lemma} \label{3pos}
As shown in Figure \ref{tu}, let $T$ and $Y$ be trees with $T\cap Y=b$ and set $D_1:=T\cup Y$. Let $U$ and $Z$ be trees with $U\cap Z=d$ and set $D_2:=U\cup Z$. Suppose $a\in T$ and $c\in U$.
Let $\psi:D_1(a)\to D_2(c)$ be an isomorphism of rooted trees.

If $|V(Z)|\gqs |V(Y)|$ then there are three possibilities:
$$
\psi(Y)\subseteq U-d, \qquad \psi(Y)\subseteq Z-d \qquad\text{or}\qquad \psi(b) = d.
$$
\end{lemma}
\begin{proof}
If $d\in \psi(Y-b)$ then $\psi^{-1}(d) \in Y-b$ and
$$
\psi^{-1}(Z) \subseteq \psi^{-1}(\rho_{D_2}(d)) = \rho_{D_1}(\psi^{-1}(d)) \subseteq Y-b.
$$
 But this is not possible as $\psi^{-1}(Z)$ has too many vertices. Hence $d\notin \psi(Y-b)$. If $\psi(b) \neq d$ then $d\notin \psi(Y)$. Since $\psi(Y)$ is connected we obtain $\psi(Y) \subseteq U-d$ or $Z-d$ completing the proof.
\end{proof}

\begin{lemma} \label{tyuz}
As shown in Figure \ref{tu}, let $T$ and $Y$ be trees with $T\cap Y=b$ and set $D_1:=T\cup Y$. Let $U$ and $Z$ be trees with $U\cap Z=d$ and set $D_2:=U\cup Z$. Suppose $a\in T$ and $c\in U$.
Let $\psi:D_1(a)\to D_2(c)$ be an isomorphism of rooted trees.

Let $\theta:T(a)\to U(c)$ also be an isomorphism of rooted trees, satisfying $\theta(b)=d$. Then $Y(b) \cong Z(d)$.
\end{lemma}
\begin{proof}
To make the proof clearer,  relabel $U$ as $T$ and $c,d$ as $a,b$ respectively. In this way $\theta$ becomes the identity map and we have the situation in Figure \ref{tu2}.
\SpecialCoor
\psset{griddots=5,subgriddiv=0,gridlabels=0pt}
\psset{xunit=0.7cm, yunit=0.7cm, runit=0.7cm}
\psset{linewidth=1pt}
\psset{dotsize=5pt 0,dotstyle=*}
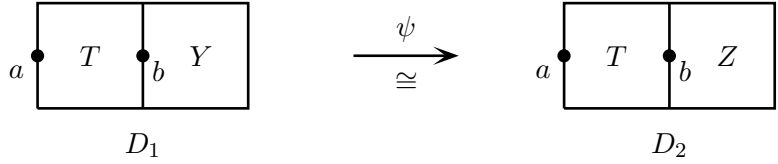
\begin{figure}[ht]
\centering
\begin{pspicture}(0,-1)(14,2.5) 

\psset{arrowscale=2,arrowinset=0.5}

\psline(0,0)(4,0)(4,2)(0,2)(0,0)
\psline(2,0)(2,2)
\psline(10,0)(14,0)(14,2)(10,2)(10,0)
\psline(12,0)(12,2)

\psline{->}(6,1)(8,1)

\psdots(0,1)(2,1)
\psdots(12,1)(10,1)

\rput(7,1.5){$\psi$}
\rput(7,0.5){$\cong$}

\rput(1,1){$T$}
\rput(3.1,1){$Y$}
\rput(2.3,0.7){$b$}
\rput(-0.4,0.7){$a$}

\rput(13.1,1){$Z$}
\rput(12.3,0.7){$b$}
\rput(9.6,0.7){$a$}
\rput(11,1){$T$}

\rput(2,-0.7){$D_1$}
\rput(12,-0.7){$D_2$}

\end{pspicture}
\caption{The isomorphic trees in Lemma \ref{tyuz}}
\label{tu2}
\end{figure}
Note that we may have $a=b$. Since $|V(Y)|=|V(Z)|$, Lemma \ref{3pos} implies that $\psi(Y) \subseteq T-b$, $\psi(Y) \subseteq Z-b$ or $\psi(b)=b$. It cannot be true that $\psi(Y) \subseteq Z-b$ as $Y$ is too large. We claim that $\psi(b)=b$ implies that $Y(b) \cong Z(b)$, proving the lemma directly. To see this, let $Y':=\rho_{T}(b)$. Then $\rho_{D_1}(b) = Y\cup Y'$ and the claim follows from
$$
\psi(Y\cup Y')=\psi(\rho_{D_1}(b)) = \rho_{D_2}(b) = Y'\cup Z.
$$
This also proves the lemma when $a=b$ since this implies $\psi(b)=b$.

We may therefore assume that $Y_2:=\psi(Y) \subseteq T-b$. Let $b_2:=\psi(b)$. Considering $Y_2$ as a subgraph of $D_1$ we may apply $\psi$ again and either $\psi(Y_2) \subseteq T-b$ or $\psi(b_2)=b$. Repeat this, with $Y_{i+1}:=\psi(Y_i)$ and $b_{i+1}=\psi(b_i)$ until  $b_{n+1}=b$ for some integer $n$. This integer $n$ must exist since the subgraphs $Y_2,\dots,Y_n$ are disjoint in $T-b$. (Otherwise, if $Y_i\cap Y_j \neq\{\}$ for $i,j$ satisfying $2\lqs i<j\lqs n$ then we may apply $\psi^{1-i}$ to get $Y\cap Y_{j-i+1} \neq\{\}$ which is not true.) It follows that $b,b_2,\dots,b_n$ are distinct vertices in $T$.

Let $Y':=\rho_{T}(b)$ as before so that $\rho_{D_1}(b) = Y\cup Y'$. Put $Y'_2:=\psi(Y')$ and we claim that $Y'_2 \subseteq T-b$. Since $b_2\in Y'_2$ and we know that $b_2 \in T-b$, the claim follows if we can show that $b\notin Y'_2$. We have
$$
b\in Y'_2 \implies \psi(b_n)\in \psi(Y') \implies b_n \in Y'.
$$
The distance from $a$ of every vertex in $Y'-b$ is greater than the distance from $a$ to $b$ which equals the distance from $a$ to $b_n$. Hence $b_n \in Y'$ implies $b_n=b$, a contradiction. This proves our claim.

Repeating this argument with $Y'_{i+1}:=\psi(Y'_i)$, we find $Y'_i \subseteq T-b$ for $2\lqs i\lqs n$ and
\begin{equation*}
  \rho_{D_1}(b_i) = Y_i \cup Y'_i \cong Y\cup Y' \qquad\text{for}\qquad 2\lqs i\lqs n.
\end{equation*}
Then
$$
 Y\cup Y' \cong \psi(Y_n\cup Y'_n) = \psi(\rho_{D_1}(b_n)) = \rho_{D_2}(b)=Z\cup Y'.
$$
Hence $Y\cup Y' \cong Z\cup Y'$ under an isomorphism that maps $b$ to $b$. This completes the proof.
\end{proof}

\section{Uniqueness of telescoping vertices} \label{unique}


\SpecialCoor
\psset{griddots=5,subgriddiv=0,gridlabels=0pt}
\psset{xunit=0.7cm, yunit=0.7cm, runit=0.7cm}
\psset{linewidth=1pt}
\psset{dotsize=5pt 0,dotstyle=*}
\begin{figure}[ht]
\centering
\begin{pspicture}(-3,-0.5)(13,3.5) 

\psset{arrowscale=2,arrowinset=0.5}

\pscircle(0.5,1.5){1.5}
\rput(0.5,1.8){odd}
\rput(0.5,1.2){cycle}
\rput(2.2,0.9){$A$}

\psline(4,2)(6,2)(6,3)(4,3)(4,2)
\psline(4,0)(6,0)(6,1)(4,1)(4,0)
\psline(6,2)(8,2)(8,3)(6,3)
\psline(6,0)(8,0)(8,1)(6,1)
\psline(4,2.5)(2,1.5)(4,0.5)

\psdots(2,1.5)(4,2.5)(4,0.5)(6,2.5)(6,0.5)

\rput(5,2.5){$T$}
\rput(3.6,2.7){$a$}
\rput(6.4,2.7){$b$}
\rput(5,0.5){$U$}
\rput(3.6,0.3){$c$}
\rput(6.4,0.3){$d$}
\rput(7.3,2.5){$X$}
\rput(7.3,0.5){$Z$}

\rput(11.5,1.5){$T \cong U$}

\end{pspicture}
\caption{The graph $G$ in Proposition \ref{onet}}
\label{telexz}
\end{figure}
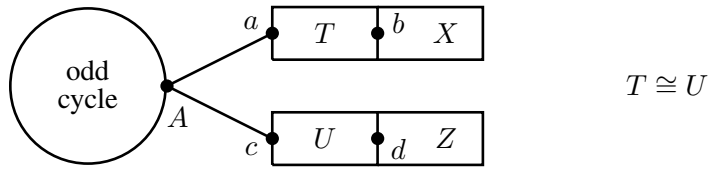

\begin{prop} \label{onet}
Suppose the tree $T$ with root $a$ is isomorphic to $U$ with root $c$. Let the tree $X$ be attached to $T$ at $b$ and the tree $Z$ attached to $U$ at the corresponding point $d$.
Connect an odd cycle containing a vertex $A$   to these trees by the edges $Aa$ and $Ac$.  Call this graph $G$, as seen in Figure \ref{telexz}, and assume that $G$ is reduced. Then any telescoping vertex of $G$  is in $X-b$ or $Z-d$.
\end{prop}
\begin{proof}
Recall from Lemma \ref{can} that, if we remove a vertex $v$ from the tree part of $G$, any cancellations that are then possible  must occur at a sequence of distinct and uniquely defined vertices in $P(A,v)-v$ that each get closer to $A$.

Assume, without losing generality, that $|V(X)|\gqs |V(Z)|$. We first claim that any $v\in U$ cannot be telescoping. Since the cancellations can only occur at points in $U$ and at $A$, it is clear that a final cancellation at $A$ will not be possible as there are too many vertices in $T\cup X$ to cancel those left from $U\cup Z$. This proves the claim.

Now we suppose that $v\in T$ is telescoping. Let $e$ be the last vertex in $T$ where cancellation occurs, or the remaining vertex adjacent to $v$ if there is no cancellation in $T$. Then we may write $T=Y\cup W$ with $v\in Y$ and $Y\cap W=e$; see the left of Figure \ref{wyz}. After removing $v$ and cancelling, we are left with $W$ from $T$ and, since $v$ is telescoping, there is one further cancellation  at $A$. Note that $P(a,b)$ must be contained in $W$ or else there will be too few vertices to cancel with $U\cup Z$ at $A$. Since $U\cong T$, the following arguments become clearer if we relabel $U$ as $T$  and the vertices $c,d$ as $a,b$ respectively. The last cancellation at $A$ requires an isomorphism $\theta$ from $D_1:=T\cup Z$ to $D_2$, which is what remains of $T\cup X$, that fixes $a$. All of $X$ will remain after the cancellation at $e$ except possibly if $e=b$, so we examine this case first.
\begin{description}
  \item[The case when $e=b$.] Suppose the cancellation at $b$ removes $Y-b$ from $T$ as well as $Y'-b$ from $X$. (If $Y'=b$ then the cancellation at $b$ does not affect $X$.) Write $X=Y'\cup X'$ with $Y'\cap X'=b$. So after the cancellation at $b$, $T\cup X$ becomes $W\cup X'$. Therefore we may write our isomorphism as $\theta:W\cup Y\cup Z \to W\cup X'$. By Lemma \ref{tyuz}, $Y\cup Z$ and $X'$ are isomorphic as trees rooted at $b$. This shows that there are two copies of $Y$ attached to $b$ in $T\cup X$ that may be cancelled. This contradicts $G$ being reduced.
\end{description}

We may assume for the remainder that $e\neq b$.
If $a=b$ then we have $\theta(b)=b$. As a consequence of this,  $X$ and $Z$ are isomorphic as trees rooted at $b$ which contradicts $G$ being reduced. Hence we may also assume that $a\neq b$.
It is possible that $a=e$.

\SpecialCoor
\psset{griddots=5,subgriddiv=0,gridlabels=0pt}
\psset{xunit=0.7cm, yunit=0.7cm, runit=0.7cm}
\psset{linewidth=1pt}
\psset{dotsize=5pt 0,dotstyle=*}
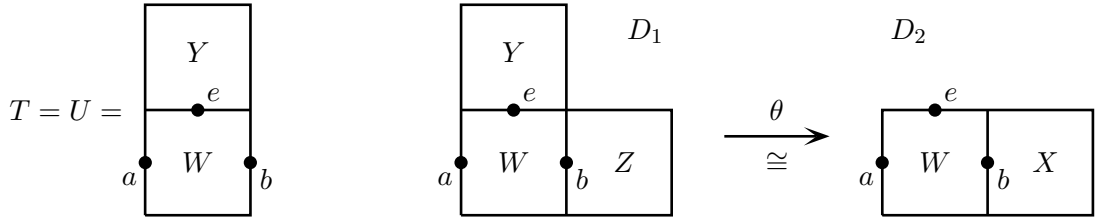
\begin{figure}[ht]
\centering
\begin{pspicture}(-3,-0.5)(18,4.5) 

\psset{arrowscale=2,arrowinset=0.5}

\psline(0,0)(2,0)(2,4)(0,4)(0,0)
\psline(0,2)(2,2)
\psline(6,0)(10,0)(10,2)(6,2)
\psline(6,0)(6,4)(8,4)(8,0)
\psline(14,0)(18,0)(18,2)(14,2)(14,0)
\psline(16,0)(16,2)

\psline{->}(11,1.5)(13,1.5)

\psdots(0,1)(1,2)(2,1)
\psdots(6,1)(7,2)(8,1)
\psdots(14,1)(15,2)(16,1)

\rput(-1.5,2){$T=U=$}
\rput(12,2){$\theta$}
\rput(12,1){$\cong$}

\rput(1,1){$W$}
\rput(1,3.1){$Y$}
\rput(1.3,2.3){$e$}
\rput(-0.3,0.7){$a$}
\rput(2.3,0.7){$b$}

\rput(7,1){$W$}
\rput(7,3.1){$Y$}
\rput(7.3,2.3){$e$}
\rput(5.7,0.7){$a$}
\rput(8.3,0.7){$b$}
\rput(9.1,1){$Z$}

\rput(15,1){$W$}
\rput(15.3,2.3){$e$}
\rput(13.7,0.7){$a$}
\rput(16.3,0.7){$b$}
\rput(17.1,1){$X$}

\rput(9.5,3.5){$D_1$}
\rput(14.5,3.5){$D_2$}

\end{pspicture}
\caption{Examining cancellation in  $G$ for Proposition \ref{onet}}
\label{wyz}
\end{figure}

The isomorphism $\theta$ from $D_1:=W\cup Y\cup Z$ to $D_2:=W\cup X$ is shown in Figure \ref{wyz}.
Recall that $T\cup X=W\cup Y\cup X$. Roughly, our goal is to show that the existence of $\theta$ means that $X$ contains a copy of $Y$, implying that $T\cup X$ has cancellation. This contradicts our assumption that $G$ was reduced and rules out $v\in T$ being a telescoping vertex. In more detail, we seek adjacent vertices $x,r\in W$ so that $\theta(x)\neq x$ and $\theta(r)=r$. Set  $T^*  :=\rho_{D_1}(x)$ which implies $\theta(T^*)  =\rho_{D_2}(\theta(x))$. We also require
\begin{align}\label{tst}
  \rho_{D_1}(x)& \subseteq Y\cup(W-b),  \\
  \rho_{D_2}(\theta(x)) &\subseteq X\cup(W-e). \label{tst2}
\end{align}
Then \eqref{tst} and \eqref{tst2} imply that $T^*=\rho_{T\cup X}(x)$ and $\theta(T^*)=\rho_{T\cup X}(\theta(x))$. If
\begin{equation}\label{empt}
  T^*\cap \theta(T^*)=\{\}
\end{equation}
then it follows that we have cancellation at $r$ in $T\cup X$, giving our desired contradiction. In fact it is easy to see that \eqref{empt} holds since otherwise there would be more than one path between $x$ and $\theta(x)$. If $e\in T^*$ then  \eqref{empt} implies \eqref{tst2}.

By Lemma \ref{3pos} we know that $\theta(Z)\subseteq W-b$, $\theta(Z)\subseteq X-b$ or $\theta(b)=b$. As we saw earlier, $\theta(b)=b$ may be ruled out. If $\theta(Z)\subseteq X-b$ then $\theta(b)\in X-b$ and so the distance from $a$ to $b$ must increase under $\theta$ which is not possible. Therefore  $\theta(Z)\subseteq W-b$ and we set $b_2:=\theta(b)$.
We may consider $Z_2 \subseteq W-b$ as a subgraph of $D_1$ and apply $\theta$ again.
As in the proof of Lemma \ref{tyuz}, there exists $n\gqs 2$ so that we obtain $Z_i \subseteq W-b$ for $2\lqs i\lqs n$ where $Z_{i+1}:=\theta(Z_i)$, $b_{i+1}=\theta(b_i)$ and $b_{n+1}=b$. We have  that $Z_1:=Z,$ $Z_2, \dots , Z_n$ are all disjoint as subgraphs of $D_1$ and $b_1:=b,b_2,\dots,b_n$ are all distinct in $W$.

Let $Z'=Z'_1:=\rho_{W\cup Y}(b)$ so that
$$
\rho_{D_1}(b)=Z\cup Z'.
$$
 As in the proof of Lemma \ref{tyuz} for $Y'_i$, we have $Z'_i \subseteq W-b$ for $2\lqs i\lqs n$ where $Z'_{i+1}:=\theta(Z'_i)$.
 Suppose $e$ is never a vertex of $Z_i\cup Z'_i$ for $1\lqs i\lqs n$. Then
\begin{equation*} 
  \rho_{D_1}(b_i) = Z_i\cup Z'_i \cong Z\cup Z' \quad \text{for}\quad 1\lqs i\lqs n.
\end{equation*}
 This being the case, we have
\begin{equation*}
    X\cup Z' \cong \rho_{D_2}(b)= \theta( \rho_{D_1}(b_n)) \cong Z\cup Z'
\end{equation*}
and so $X \cong Z$ as trees rooted at $b$.
However this contradicts $G$ being reduced and
so we must have
 $e\in Z_k\cup Z'_k$ for some $k$ in the range $1\lqs k\lqs n$.
\begin{description}
  \item[The case when $e\in Z_1\cup Z'_1$.] Clearly, $e\notin Z_1$ so assume that $e\in Z'_1$.
Let $R:=\rho_{W}(b)$ so that $Z'_1=R\cup Y$ with $R\cap Y=e$. By applying $\theta$ $n$ times to $\rho_{D_1}(b)$ we obtain
\begin{equation} \label{arpp}
  R\cup Y\cup Z \cong R\cup X
\end{equation}
as trees rooted at $b$. Then $P(b,e)\subseteq R$ and we let $x$ be the vertex adjacent to $b$ in $P(b,e)$. Put $R_1:=bx \cup \rho_W(x) \subseteq R$. This makes $R_1$ the tree in $W$ containing $e$ that is attached to  $b$ by a single edge. A short argument using \eqref{arpp} shows that $X$ must contain a copy of $R_1\cup Y$. In other words $X=R_1\cup Y \cup X_1$ with $(R_1\cup Y)\cap X_1=b$. We have $\rho_{T\cup X}(b)=R\cup Y\cup X$ which means there are two copies of $R_1\cup Y$ attached to $b$ in $T\cup X$. So in this case we have cancellation at $b$ in $G$, contradicting our assumption that $G$ was reduced.
\end{description}

It remains to suppose that   $e\in Z_k\cup Z'_k$ only for  $k$ in the range $2\lqs k\lqs n$. We choose the minimal such $k$. Then
\begin{equation*} 
  \rho_{D_1}(b_k) = Y \cup Z_k\cup Z'_k.
\end{equation*}
 The distinct vertices $b_1,\dots, b_n$ are  in $W$ and so  $P(a,b_1),\dots, P(a,b_n)$ are in $W$ with $$
\theta(P(a,b_i))=P(a,b_{i+1}) \quad \text{for}\quad 1\lqs i\lqs n.
$$
Write $P(a,b_1) \cap P(a,b_2)$ as $P(a,r)$ so that $\theta(r)=r$. Then $P(a,r) \subseteq  P(a,b_i)$ for $1 \lqs i\lqs n$.
Now let $x$ be the vertex in $P(r,b_k)$ adjacent to $r$. Put $T^*:=\rho_{D_1}(x)$ and we claim that $T^* \subseteq Y\cup(W-b)$. If $b\in T^*$ then there is a path from $b_k$ to $b$ in $T^*$ that  does not pass through $r$. But this is not possible as the path $P(b_k,r)\cup P(r,b)$ does pass through $r$. Since $T^*$ is connected and contains a vertex $b_k$ in $W-b$, we have proved the claim and so \eqref{tst} holds. By construction we have $e\in T^*$, giving \eqref{tst2} by \eqref{empt}. As discussed there, these conditions prove that $G$ has cancellation at $r$ giving a contradiction that implies $v\in T$ cannot be a telescoping vertex.
\end{proof}

Proposition \ref{onet} is next extended to cases with more cancellation.

\SpecialCoor
\psset{griddots=5,subgriddiv=0,gridlabels=0pt}
\psset{xunit=0.7cm, yunit=0.7cm, runit=0.7cm}
\psset{linewidth=1pt}
\psset{dotsize=5pt 0,dotstyle=*}
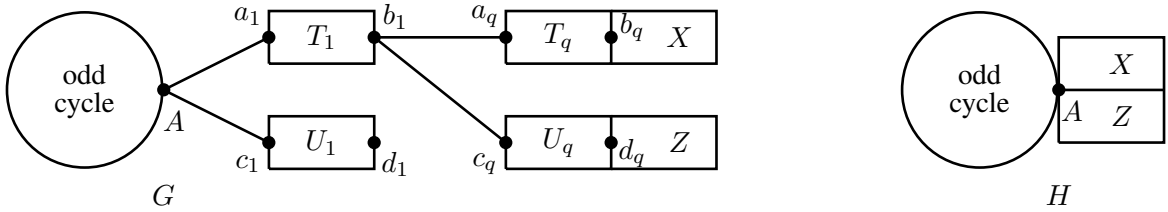
\begin{figure}[ht]
\centering
\begin{pspicture}(-1,-0.5)(21,3.5) 

\psset{arrowscale=2,arrowinset=0.5}

\pscircle(0.5,1.5){1.5}
\rput(0.5,1.8){odd}
\rput(0.5,1.2){cycle}
\rput(2.2,0.9){$A$}

\psline(4,2)(6,2)(6,3)(4,3)(4,2)
\psline(4,0)(6,0)(6,1)(4,1)(4,0)
\psline(8.5,2)(10.5,2)(10.5,3)(8.5,3)(8.5,2)
\psline(8.5,0)(10.5,0)(10.5,1)(8.5,1)(8.5,0)
\psline(10.5,2)(12.5,2)(12.5,3)(10.5,3)
\psline(10.5,0)(12.5,0)(12.5,1)(10.5,1)

\psline(4,2.5)(2,1.5)(4,0.5)
\psline(8.5,2.5)(6,2.5)(8.5,0.5)

\psdots(2,1.5)(4,2.5)(4,0.5)(6,2.5)(6,0.5)(8.5,2.5)(8.5,0.5)(10.5,2.5)(10.5,0.5)

\rput(5,2.5){$T_1$}
\rput(3.6,2.9){$a_1$}
\rput(6.4,2.9){$b_1$}
\rput(5,0.5){$U_1$}
\rput(3.6,0.1){$c_1$}
\rput(6.4,0.1){$d_1$}

\rput(9.5,2.5){$T_q$}
\rput(8.1,2.9){$a_q$}
\rput(10.9,2.7){$b_q$}
\rput(9.5,0.5){$U_q$}
\rput(8.1,0.1){$c_q$}
\rput(10.9,0.3){$d_q$}

\rput(11.8,2.5){$X$}
\rput(11.8,0.5){$Z$}


\rput(20.2,2){$X$}
\rput(20.2,1){$Z$}

\pscircle(17.5,1.5){1.5}
\rput(17.5,1.8){odd}
\rput(17.5,1.2){cycle}
\rput(19.25,1.1){$A$}
\rput(19,-0.5){$H$}
\rput(2,-0.5){$G$}

\psline(19,0.5)(21,0.5)(21,2.5)(19,2.5)(19,0.5)
\psline(19,1.5)(21,1.5)

\psdots(19,1.5)

\end{pspicture}
\caption{The graphs for Proposition \ref{onet2}}
\label{more}
\end{figure}

\begin{prop} \label{onet2}
For $i=1,2,\dots,q$, let the trees $T_i$ and $U_i$ be isomorphic with vertices $a_i,b_i\in T_i$ corresponding to $c_i,d_i\in U_i$ respectively. Include edges $b_ia_{i+1}$ and $b_ic_{i+1}$ for $i=1,2,\dots,q-1$. Let an odd cycle containing the vertex $A$ be connected with edges $Aa_1$ and $Ac_1$. Attach the tree $X$ at $b_q$ and the tree $Z$ at $d_q$. The left of Figure \ref{more} shows the $q=2$ case of this construction. Call this graph $G$ and assume it  is reduced with a telescoping vertex $t$.

Then the following are true:
\begin{enumerate}
  \item The vertex $t$ is in $X-b_q$ or $Z-d_q$.
  \item Make a new graph $H$ from $G$ by removing all the trees $T_i,U_i$ and identifying the vertices $A,b_q,d_q$ as displayed on the right of Figure \ref{more}. Then $t$ is a telescoping vertex for $H$.
\end{enumerate}
\end{prop}
\begin{proof}
We use induction on $q$ to prove (i). The $q=1$ case is covered by Proposition \ref{onet} so assume $q\gqs 2$.  Proposition \ref{onet} implies that $t$ is not in $T_1$ or $U_1$. Suppose that the first cancellation after $t$ is removed happens at $s_m$, the next at $s_{m-1}$ and the last at $s_0=A$. By Lemma \ref{can}, we know that every $s_i$ is in $P(A,t)-t$. The last cancellation before $A$ is at $s_1$. We cannot have $s_1$ in $T_1-b_1$ as there will be too few vertices remaining to cancel with $U_1$ at $A$. If $s_1$ is not in $T_1$ then there will be too many vertices remaining to cancel with $U_1$ at $A$. Therefore $s_1=b_1$. Let $G^*$ be $G$ with $T_1,U_1$ removed and $b_1$ identified with $A$. Induction now shows that $t$ is in $X-b_q$ or $Z-d_q$.

We next show (ii). Repeating the above argument that $s_1=b_1$ also shows that $s_2=b_2, \dots, s_{q-1}=b_{q-1}$. Assume  $|V(X)|\gqs |V(Z)|$ and, as at the beginning of the proof of Proposition \ref{onet}, this implies that $t\in X-b_q$ and also that the cancellation before $s_{q-1}=b_{q-1}$ (or before $A$ if $q=1$) occurs at vertices in $X-b_q$. Suppose the last cancellation before $b_{q-1}$ is at the vertex $s\in X-b_q$. We may write $X=Y\cup X'$ for $Y\cap X'=s$ so that the cancellation at $s$ removes $X'-s$. To cancel at $b_{q-1}$ we must have $U_q \cup Z$ isomorphic to $T_q\cup Y$ as trees rooted at $c_q$ and $a_q$ respectively. Then Lemma \ref{tyuz} implies that $Y$ rooted at $b_q$ is isomorphic to $Z$ rooted at $d_q$. It follows that $t$ is a telescoping vertex for $H$.
\end{proof}

We will need Proposition \ref{onet2} to prove Proposition \ref{ned} in the next section. It also allows us to prove the goal of this section:

\begin{cor} \label{onetel}
Let $G$ be a reduced graph consisting of a tree attached to an odd cycle at one vertex. Then $G$ has at most one telescoping vertex.
\end{cor}
\begin{proof}
Suppose $G$ has at least one telescoping vertex. We may describe $G$ as in Figure \ref{tele} and its related discussion, with telescoping vertex $a_{q+1}$. By Proposition \ref{onet2} part (i), applied with $X=b_qa_{q+1}\cup T_{q+1}$ and $Z=d_q$, any telescoping vertex  $t$ of $G$ must be in $T_{q+1}$. Lemma \ref{teleprops} part (ii) implies that $t$ must be the same distance from $A$ as $a_{q+1}$. Hence $t=a_{q+1}$.
\end{proof}

\section{Nim-values of  graphs with one odd cycle} \label{1cyc}

The following result is Theorem 4 of \cite{Y}. We reproduce their proof in a slightly shorter form.

\SpecialCoor
\psset{griddots=5,subgriddiv=0,gridlabels=0pt}
\psset{xunit=0.7cm, yunit=0.7cm, runit=0.7cm}
\psset{linewidth=1pt}
\psset{dotsize=5pt 0,dotstyle=*}
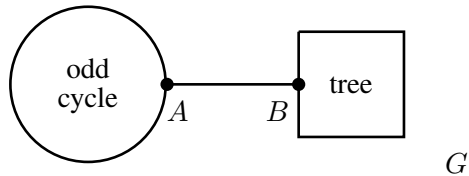
\begin{figure}[ht]
\centering
\begin{pspicture}(5,-4.5)(13.5,-0.6) 

\psset{arrowscale=2,arrowinset=0.5}

\pscircle(6.5,-2.5){1.5}
\psline(10.5,-1.5)(12.5,-1.5)(12.5,-3.5)(10.5,-3.5)(10.5,-1.5)
\psline(8,-2.5)(10.5,-2.5)
\psdots(8,-2.5)(10.5,-2.5)

\rput(11.5,-2.5){tree}
\rput(10.1,-3){$B$}

\rput(6.5,-2.2){odd}
\rput(6.5,-2.8){cycle}
\rput(8.2,-3){$A$}
\rput(13.5,-4){$G$}

\end{pspicture}
\caption{A tree attached to an odd cycle}
\label{fab}
\end{figure}
\begin{theorem} \label{AB}
Let $G$ be an odd cycle connected to a tree by an edge $AB$ as shown in Figure \ref{fab}. Then
\begin{equation*}
  \n(G)= \left\{
           \begin{array}{ll}
             4 \text{ or more} & \hbox{if $B$ has odd degree;} \\
             \f(G) & \hbox{otherwise.}
           \end{array}
         \right.
\end{equation*}
\end{theorem}
\begin{proof}
Let $n$ be the number of vertices and edges of $G$ outside the cycle. The proof uses induction on $n$ and the base case  with $n=2$ may be easily verified. The following cases establish our argument.
\begin{description}
  \item[{\it Case (i).}]
 Suppose the degree of $B$ is odd. Removing a degree $2$ vertex or an edge from the cycle gives  trees with nim-values $1$ and $2$. Removing vertices $A$ and $B$ gives nim-values $0$ and $3$. It follows that the nim-value of $G$ is greater than $3$.

\item[{\it Case (ii).}] Suppose the degree of $B$ is even. We have $\f(G)=0$ or $3$ and want to show that this equals the nim-value of $G$. Removing a degree $2$ vertex or an edge from the cycle gives  trees with nim-values $1$ and $2$. Removing $B$ or the edge $AB$  also yields nim-values $1$ or $2$. Deleting $A$ gives nim-value $3$ if $\f(G)=0$ and  nim-value $0$ if $\f(G)=3$. To complete the proof it remains to show that no other move $H$ gives nim-value $\f(G)$. By the induction hypothesis $g(H) \gqs 4$ or $g(H)=\f(H)$ and clearly $\f(H)\neq \f(G)$. \qedhere
\end{description}
\end{proof}

Theorem \ref{AB2} will generalize Theorem \ref{AB}. It requires the next definition and proposition.

\begin{adef}{\rm
For   trees $X$ and $Z$, we call the tree constructed in the statement of Proposition \ref{onet2} an {\em $(X,Z)$ tree of level $q \gqs 1$}.  If $X$ or $Z$ is just a single vertex then we may replace them by $\bullet$ in the notation.
}
\end{adef}

\begin{prop} \label{ned}
Let $G$ be a reduced graph consisting of an odd cycle connected at vertex $A$ to a tree. Suppose $A$ has even degree at least $4$. Then there exists an odd degree vertex of the tree that is not telescoping and so that removing it and reducing produces a subgraph with  no telescoping vertices of odd degree.
\end{prop}
\begin{proof} 
As in Lemma \ref{teleprops}, label the set of vertices of the tree a distance $x$ from $A$ as $S_x$. Let $\ell$ be the largest $x$ for which $|S_x|>0$. In other words, $\ell$ is the height of the tree rooted at $A$. We use induction on $\ell$. Since $G$ is reduced, the base case has $\ell=2$ with the tree in $G$ consisting of paths of length $1$ and $2$ attached to $A$. The proposition is clearly true in this case. Now assume $\ell\gqs 3$.

For some $d$ with $2\lqs d \lqs \ell$, suppose that the numbers $|S_1|, |S_2|, \dots, |S_{d-1}|$ are even and $|S_d|$ is odd. Then the total degree of the vertices in $S_{d-1}$ is odd and hence there exists a vertex $t\in S_{d-1}$ of odd degree. Remove vertex $t$ and reduce. Note that  $t \not\in S_d$ and so cannot be telescoping by Lemma \ref{teleprops} part (ii). Also from Lemma \ref{teleprops} part (ii), if the resulting graph has a telescoping vertex $v$ then it must be one of the remaining vertices in $S_{d-1}-t$. But it cannot be a telescoping vertex of odd degree since that would contradict Lemma \ref{teleprops} part (iii) which says that the total degree of $(S_{d-1}-t)-v$ is even. Therefore, $t$ is the desired vertex.

\SpecialCoor
\psset{griddots=5,subgriddiv=0,gridlabels=0pt}
\psset{xunit=0.7cm, yunit=0.7cm, runit=0.7cm}
\psset{linewidth=1pt}
\psset{dotsize=5pt 0,dotstyle=*}
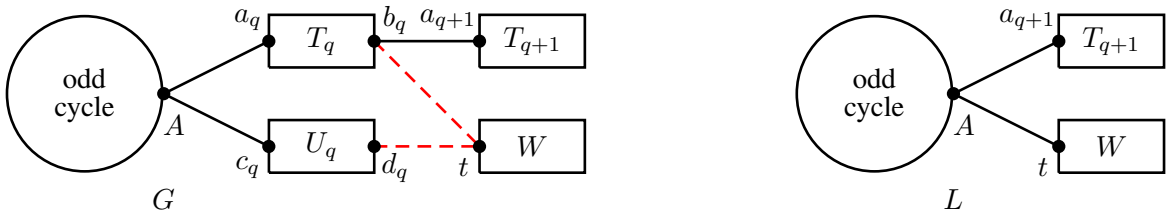
\begin{figure}[ht]
\centering
\begin{pspicture}(-1,-0.9)(21,3.9) 

\psset{arrowscale=2,arrowinset=0.5}

\pscircle(0.5,1.5){1.5}
\rput(0.5,1.8){odd}
\rput(0.5,1.2){cycle}
\rput(2.2,0.9){$A$}

\psline(4,2)(6,2)(6,3)(4,3)(4,2)
\psline(4,0)(6,0)(6,1)(4,1)(4,0)
\psline(8,2)(10,2)(10,3)(8,3)(8,2)
\psline(8,0)(10,0)(10,1)(8,1)(8,0)

\psline(4,2.5)(2,1.5)(4,0.5)
\psline(8,2.5)(6,2.5)
\psline[linestyle=dashed,linecolor=red](6,2.5)(8,0.5)(6,0.5)


\psdots(2,1.5)(4,2.5)(4,0.5)(6,2.5)(6,0.5)(8,2.5)(8,0.5)

\rput(5,2.5){$T_q$}
\rput(3.6,2.9){$a_q$}
\rput(6.4,2.9){$b_q$}
\rput(5,0.5){$U_q$}
\rput(3.6,0.1){$c_q$}
\rput(6.4,0.1){$d_q$}

\rput(9,2.5){$T_{q+1}$}
\rput(7.4,2.9){$a_{q+1}$}
\rput(9,0.5){$W$}
\rput(7.7,0.1){$t$}

\rput(20,0.5){$W$}
\rput(18.7,0.1){$t$}
\rput(20,2.5){$T_{q+1}$}
\rput(18.4,2.9){$a_{q+1}$}

\pscircle(15.5,1.5){1.5}
\rput(15.5,1.8){odd}
\rput(15.5,1.2){cycle}
\rput(17.2,0.9){$A$}

\psline(19,2)(21,2)(21,3)(19,3)(19,2)
\psline(19,0)(21,0)(21,1)(19,1)(19,0)
\psline(19,2.5)(17,1.5)(19,0.5)

\psdots(17,1.5)(19,2.5)(19,0.5)

\rput(17,-0.5){$L$}
\rput(2,-0.5){$G$}

\end{pspicture}
\caption{Graphs in the proof of Proposition \ref{ned}}
\label{tein}
\end{figure}

Otherwise, $|S_1|, |S_2|, \dots, |S_{\ell}|$ are all even, as we now assume. Remove an odd degree vertex $t$ in $S_m$ for some $m$ with $3\lqs m \lqs \ell$. For example $t$ could be a leaf. After removing $t$ suppose there are $p\gqs 0$ cancellations to obtain the reduced graph $H$. In general $t$ is contained in the subgraph $W:=\rho_{G}(t)$ so that $W-t$ becomes disconnected from the cycle component $H$. Note that  $t$ cannot be telescoping by Lemma \ref{teleprops} part (ii). If  $H$ has no telescoping vertices of odd degree then we are done. Otherwise, the telescoping vertex is in $S_m$ by Lemma \ref{teleprops} part (ii). We see that $H$ must take the form shown in Figure \ref{tele} with $q\gqs 1$ since $m>1$. The telescoping vertex is labelled as $a_{q+1}$ with odd degree.

The degrees of vertices $b_q$ and $d_q$ in that figure must have opposite parity in $H$. They will have opposite parity in $G$ as well unless $p=0$ and $t$ is connected by an edge to one of $b_q$ or $d_q$. This is illustrated on the left of Figure \ref{tein}, in the case of $q=1$, with dashes denoting the possible  edges. We need to examine this situation in detail.
\begin{description}
  \item[The case when $t$  is attached to $b_q$ or $d_q$.] We  find a vertex $r$ satisfying the proposition in this case. Consider the graph $L$ on the right of Figure \ref{tein}. It is obtained from $G$ by removing $T_i,U_i$ for $1\lqs i\lqs q$ and adding the edges $Aa_{q+1}$ and $At$. Then $L$ is reduced since $G$ is. By induction, $L$ has an odd  degree vertex $r$ of the tree that is not telescoping, and so that removing it and reducing produces a subgraph $L'$ with  no telescoping vertices of odd degree.  Let $X$ be what remains of $Aa_{q+1}\cup T_{q+1}$ after this reduction and  $Z$ be what remains of $At\cup W$. Then $L'=X\cup Z$. If $r\in T_{q+1}$ then all the cancellation is in $T_{q+1}$ and $Z=At\cup W$. Similarly,  If $r\in W$ then all the cancellation is in $W$ and $X=Aa_{q+1}\cup T_{q+1}$.

Now we remove the same vertex $r$ from $G$. The initial cancellations in $G$ after removing $r$  will be the same as those that produced $L'$ and we carry them out in $G$ to obtain $G'$. Depending on where $t$ is connected, the tree part of $G'$ is an $(X\cup Z, \bullet)$ graph or an  $(X,Z)$ graph.  If $G'$ is a reduced graph then it follows from Proposition \ref{onet2} part (ii) that if $G'$ has a telescoping vertex then so does $L'$ and it must be the same vertex. Therefore $L'$ having  no telescoping vertices of odd degree implies that $G'$ does not have them either and so $r$ is the desired vertex in this case.

However, it may be the case that further cancellation is possible in $G'$.  This cancellation must be  at vertices in $P(A,r)$, though not at the vertex  $A$ since $r$ is not telescoping by Proposition \ref{onet2} part (ii). Assume, without losing generality, that $r\in T_{q+1}$ so that any cancellation happens in $T_1, \dots,T_q$. Suppose that after cancellation in  $T_{q+1}$ the next cancellation is at $s \in T_m$.
\SpecialCoor
\psset{griddots=5,subgriddiv=0,gridlabels=0pt}
\psset{xunit=0.7cm, yunit=0.7cm, runit=0.7cm}
\psset{linewidth=1pt}
\psset{dotsize=5pt 0,dotstyle=*}
\begin{figure}[ht]
\centering
\begin{pspicture}(-1,-0.5)(13,7.5) 

\psset{arrowscale=2,arrowinset=0.5}

\psline[linecolor=orange](8,6)(12,6)(12,7)(8,7)
\psline[linecolor=orange](2,5.5)(0,5.5)(2,1.5)

\psline(2,1)(4,1)(4,2)(2,2)(2,1)
\psline(2,5)(4,5)(4,6)(2,6)(2,5)
\psline(6,0)(8,0)(8,1)(6,1)(6,0)
\psline(6,2)(8,2)(8,3)(6,3)(6,2)
\psline(6,4)(8,4)(8,5)(6,5)(6,4)
\psline(6,6)(8,6)(8,7)(6,7)(6,6)

\psline(8,0)(12,0)(12,1)(8,1)
\psline(8,4)(12,4)(12,5)(8,5)

\psline(6,6.5)(4,5.5)(6,4.5)
\psline(6,2.5)(4,1.5)(6,0.5)

\psdots(2,5.5)(4,5.5)(6,6.5)(8,6.5)(6,4.5)(8,4.5)
\psdots(2,1.5)(4,1.5)(6,2.5)(8,2.5)(6,0.5)(8,0.5)
\psdot[linecolor=orange](0,5.5)

\rput(14,4.5){$T_m$}
\rput(14,0.5){$U_m$}
\rput(-0.99,5.5){{\color{orange} $\cdots\cdots$}}
\rput(0,6){{\color{orange} $b_{m-1}$}}
\rput(1.5,5){$a_{m}$}
\rput(3,5.5){$T_{m,3}$}
\rput(7,6.5){$T_{m,1}$}
\rput(7,4.5){$T_{m,2}$}
\rput(8.6,6.5){$b_{m}$}
\rput(10,6.5){{\color{orange} $W_1$}}
\rput(10,4.5){$W_2$}
\rput(4.3,4.9){$s$}

\rput(3,1.5){$U_{m,3}$}
\rput(7,2.5){$U_{m,1}$}
\rput(7,0.5){$U_{m,2}$}
\rput(1.5,1){$c_{m}$}
\rput(8.6,2.5){$d_{m}$}
\rput(10,0.5){$W_3$}

\end{pspicture}
\caption{Studying cancellation at $s$}
\label{end}
\end{figure}
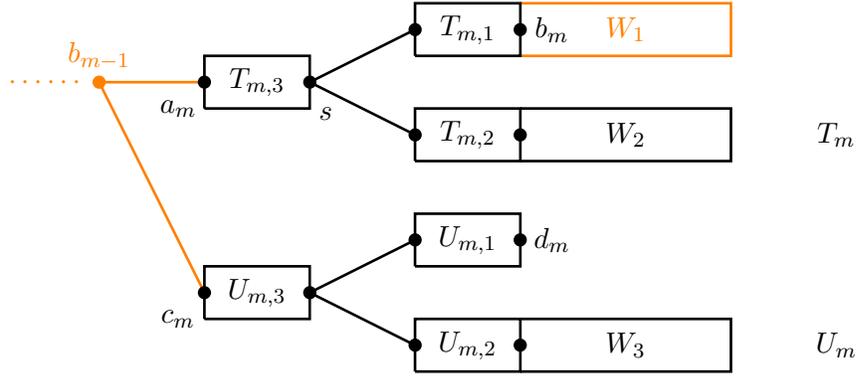
If $s\neq b_m$, then as shown in Figure \ref{end}, $T_m$ contains the  parts $T_{m,1}$ and $T_{m,2}$, which will cancel, and $T_{m,3}$ which contains $s$. Suppose that $b_m \in T_{m,1}$. Let $W_1$ be the part of $G'$ attached to $b_m$ containing $T_{m+1},U_{m+1}, \dots, X,Z$. Then $W_1$ is an $(X\cup Z, \bullet)$ graph or an  $(X,Z)$ graph. For cancellation at $s$, $T_m$ must also contain $W_2$ which is isomorphic to $W_1$ and attached to $T_{m,2}$ in the same way. Now $U_m \cong T_m$ and so has the same components $U_{m,1},$ $U_{m,2},$  $U_{m,3}$ and $W_3$ with $d_m\in U_{m,1}$ and $W_3\cong W_1$ attached to $U_{m,2}$. Therefore we see that after cancellation at $s$ we obtain a graph whose tree part is still an $(X\cup Z, \bullet)$ graph or an  $(X,Z)$ graph. Similar reasoning gives the same conclusion when $s = b_m$.

Repeating this argument shows that the reduced version, $G^*$, of $G'$ is still an $(X\cup Z, \bullet)$ graph or an  $(X,Z)$ graph. By Proposition \ref{onet2}, any telescoping vertex of $G^*$ must be in $X$ or $Z$ and a telescoping vertex for $L'$. Therefore $L'$ having  no telescoping vertices of odd degree implies that $G^*$ does not have them either and so $r$ is the desired vertex as before.
\end{description}

Returning to our main argument, we are left with the situation that one of $b_q$ and $d_q$ has odd degree in $G$. We have proved that there are three alternatives: the original odd degree vertex $t \in S_m$ we chose satisfies the proposition, the vertex $r$ does, or else there exists a new odd degree vertex in $S_{m-1}$, one unit closer to $A$.

\SpecialCoor
\psset{griddots=5,subgriddiv=0,gridlabels=0pt}
\psset{xunit=0.7cm, yunit=0.7cm, runit=0.7cm}
\psset{linewidth=1pt}
\psset{dotsize=5pt 0,dotstyle=*}
\begin{figure}[ht]
\centering
\begin{pspicture}(-1,-2.5)(14,5) 

\psset{arrowscale=2,arrowinset=0.5}

\pscircle(0.5,1.5){1.5}
\rput(0.5,1.8){odd}
\rput(0.5,1.2){cycle}
\rput(2.1,0.7){$A$}

\psline(4,2)(6,2)(6,3)(4,3)(4,2)
\psline(4,0)(6,0)(6,1)(4,1)(4,0)
\psline(4,4)(6,4)(6,5)(4,5)(4,4)
\psline(4,-2)(6,-2)(6,-1)(4,-1)(4,-2)
\psline(8,4)(10,4)(10,3)(8,3)(8,4)
\psline(8,0)(10,0)(10,-1)(8,-1)(8,0)

\psline(8,3.5)(4,4)(2,1.5)(4,2)
\psline(8,-0.5)(4,-1)(2,1.5)(4,1)

\psdots(2,1.5)(4,2)(4,1)(4,4)(4,-1)(8,3.5)(8,-0.5)

\rput(9,3.5){$T_2$}
\rput(5,4.5){$T_1$}
\rput(5,2.5){$U_1$}
\rput(3.6,4.3){$a_1$}
\rput(3.6,2.3){$c_1$}

\rput(3.6,0.7){$x$}
\rput(3.6,-1.3){$y$}

\rput(5,0.5){$X$}
\rput(5,-1.5){$Y$}
\rput(7.6,3.1){$a_2$}
\rput(7.6,-0.1){$t$}

\rput(9,-0.5){$W$}

\rput(13,2.5){$T_1(a_1) \cong U_1(c_1)$}
\rput(13,0.5){$X(x) \cong Y(y)$}

\end{pspicture}
\caption{A possible configuration in the proof of Proposition \ref{ned}}
\label{a6}
\end{figure}
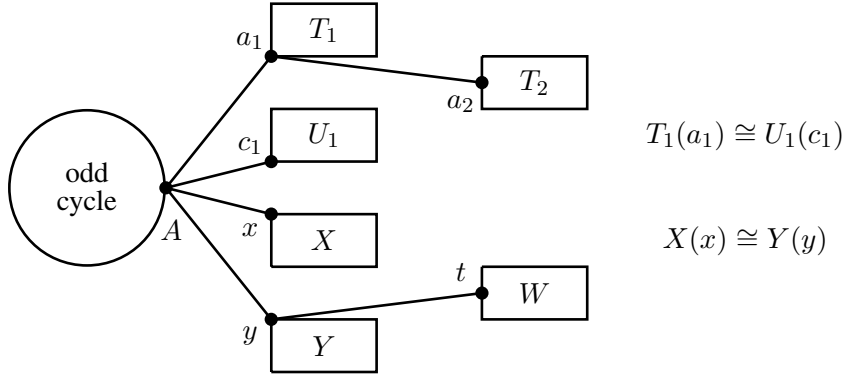

Repeating this reasoning, we eventually find a vertex satisfying the proposition or else we find an odd degree vertex $t$ in $S_2$ so that removing $t$ and reducing (with $p$ cancellations) produces a graph $H$ with an odd degree telescoping vertex $a_2 \in S_2$ (the number of cancellations in $H$ is necessarily $q=1$ and in our notation we have $a_1=b_1$ and $c_1=d_1$). Since $t \in S_2$, the only options for $p$ are $0$ and $1$. If $p=1$ then we have the situation in Figure \ref{a6} with $\deg A =6$. 
Then $S_1$ contains four vertices and two  have odd degree. Deleting one of these odd degree vertices leaves a reduced graph where $\deg A =5$. It follows from part (i) of Lemma \ref{teleprops}  that this graph does not have a telescoping vertex. This completes the $p=1$ case.

Lastly, when $p=0$ we necessarily have $t$ adjacent to $b_1$ or $d_1$. This is the highlighted case we covered earlier in the proof, and the argument there shows that the desired vertex $r$ may be found by induction.
\end{proof}

\begin{proof}[Proof of Theorem \ref{AB2}]
Let $n$ be the number of vertices and edges of $G$ outside the cycle. The proof uses induction on $n$ and the base case with $n=0$ is clearly true. Assume $n\gqs 1$ and
let $A$ be the vertex on the cycle with degree greater than $2$. If $\deg A =3$ then denote by $B$ its adjacent vertex in the tree. Clearly $B$ is a telescoping vertex for $G$ and, with Lemma \ref{teleprops} part (ii) or Corollary \ref{onetel}, it is the only one. The result then follows from Theorem \ref{AB}. Hence we may assume that $A$ has degree at least $4$.

\begin{description}
  \item[{\it Case (i).}]
 Suppose $G$ contains a telescoping vertex of odd degree. Removing a degree $2$ vertex or an edge from the cycle gives  trees with nim-values $1$ and $2$. Removing the telescoping vertex leaves a graph $H$, consisting of a cycle component and a disconnected forest, and with $\f(H)=3\oplus\f(G)$ because we removed a vertex of odd degree. Since the cycle component has $\f=3$ and $\n=0$ it follows that $\n(H)=\f(G)$ which is $0$ or $3$. Next we note that $\deg A = 4$ by Lemma \ref{teleprops} part (i). Proposition \ref{ned} then implies there  exists an odd degree vertex in the tree part of $G$ so that deleting it and reducing gives a graph $H'$ that does not contain a telescoping vertex of odd degree and in which $\deg A \gqs 3$. By induction $\n(H')=\f(H')$ and so this move gives the fourth element of the set $\{0,1,2,3\}$. Therefore the nim-value of $G$ is at least $4$.

\item[{\it Case (ii).}]  Suppose $G$ does not contain a telescoping vertex of odd degree.
We have $\f(G)=0$ or $3$ and want to show that this equals the nim-value of $G$. Removing a degree $2$ vertex or an edge from the cycle gives  trees with nim-values $1$ and $2$. Deleting any vertex or edge from the tree part of $G$ and reducing gives a graph $H$ in which $\deg A \gqs 3$ and that by induction has nim-value $\f(H)$ or at least $4$.
If
$\f(G)=0$ then it follows that all moves have nim-value $\neq 0$ and so $\n(G)=0$. If $\f(G)=3$ then it follows that all moves have nim-value $\neq 3$. To show that $\n(G)=3$ it remains to find a move with nim-value $0$. If $A$ has odd degree then removing it is such a move. If $A$ is even then Proposition \ref{ned}  implies there  exists an odd degree vertex in the tree part of $G$ so that deleting it and reducing gives a graph $H'$ that does not contain a telescoping vertex of odd degree and in which $\deg A \gqs 3$. By induction $\n(H')=\f(H')$ and so we have located the required nim-value $0$ move. \qedhere
\end{description}
\end{proof}

We may now prove  Theorem \ref{conj}, as conjectured  in \cite{Y}. 

\begin{proof}[Proof of Theorem \ref{conj}]
In general, the graph $G$  consists of trees attached to an odd cycle and possibly a number of other disjoint trees.
Let $n$ be the number of vertices and edges of $G$ that are  not on the cycle. The proof uses induction on $n$ and the base case with $n=4$ is easy to verify.
If $\n(G)=\f(G)$ then, with Proposition \ref{par} and \eqref{fff}, we obtain the same relation if we add disjoint trees to $G$. Hence
we may assume that $G$ is connected. Then $\f(G)=0$ or $3$ and want to show that this equals the nim-value of $G$.
The following cases establish the  argument.
\begin{description}
  \item[{\it Case (i).}]  Suppose $\f(G)=0$. This corresponds to $n$ being odd. Removing a vertex or edge from the cycle leaves a non-zero nim-value by Proposition \ref{par}. Removing a vertex or edge not on the cycle and reducing leaves a graph $H$ where $m$ vertices of the cycle have degree greater than $2$ for $m\gqs 1$. (There cannot be any cancellation between trees attached to different vertices of the cycle; see the discussion after Lemma \ref{can}.) If $m\gqs 2$ then $\n(H)=\f(H)$ by induction. If $m=1$ then $\n(H)=\f(H)$ or $\n(H)\gqs 4$ by Theorem \ref{AB2}. In either case $\n(H)\neq 0$ and so $\n(G)=0$.

  \item[{\it Case (ii).}]  Suppose $\f(G)=3$. This corresponds to $n$ being even.  Similar arguments to Case (i) show that all moves of $G$ have nim-value $\neq 3$. It remains to show that moves with nim-values $0,$ $1$ and $2$ exist. Removing an edge of the cycle gives nim-value $1$. Deleting an odd degree vertex on the cycle gives nim-value $0$ and deleting an even degree vertex on the cycle gives nim-value $2$. If all the cycle vertices have odd degree then there must exist an even degree vertex in the rest of $G$ (since $G$ has an odd number of vertices) and removing it gives  nim-value $2$ by induction. In the last case to consider, all the cycle vertices have even degree and we need to locate a move with nim-value $0$. Choose a cycle vertex $A$ with even degree at least $4$.  Proposition \ref{ned}  implies there  exists an odd degree vertex in the tree attached to $A$ so that deleting it and reducing gives a graph $H'$  in which $\deg A \gqs 3$. By induction $\n(H')=\f(H')$ and this provides the  nim-value $0$ move we wanted.  \qedhere
  \end{description}
\end{proof}

\begin{cor}
Let $G$ be a reduced, possibly disconnected graph containing  one odd cycle and no other cycles or loops.  Then $\n(G) \neq \f(G)$ if and only if one of these conditions is true:
\begin{enumerate}
  \item all of the cycle vertices have degree $2$ or
  \item exactly one of the cycle vertices has degree greater than $2$ and there exists a telescoping vertex of odd degree.
\end{enumerate}
\end{cor}

The results we have proved in this section are for graphs with exactly one odd cycle and no  even cycles or loops. It is natural to also ask what happens when we allow even cycles or if we replace the odd cycle with a loop.

\section{Unbounded nim-values} \label{unbd}

\SpecialCoor
\psset{griddots=5,subgriddiv=0,gridlabels=0pt}
\psset{xunit=0.7cm, yunit=0.7cm, runit=0.7cm}
\psset{linewidth=1pt}
\psset{dotsize=5pt 0,dotstyle=*}

\begin{figure}[ht]
\centering
\begin{pspicture}(-2,0.5)(18,3.5) 

\psset{arrowscale=2,arrowinset=0.5}

\psdots(0,1)(0,3)(2,2)(4,2)(-2,2)(6,2)(8,2)
\psline(6,2)(8,2)
\psdots(12,2)(14,2)(16,2)(18,2)

\psline(0,1)(0,3)(2,2)(4,2)(2,2)(0,1)(-2,2)(0,3)
\psline(12,2)(14,2)
\psline(16,2)(18,2)
\pscurve(12,2)(11,2.5)(11,1.5)(12,2)

\rput(5.05,2){$\cdots\cdots$}
\rput(15.05,2){$\cdots\cdots$}

\rput(3,1){$G_1(n)$}
\rput(15,1){$G_2(n)$}

\end{pspicture}
\caption{Graphs with unbounded nim-values}
\label{unb12}
\end{figure}
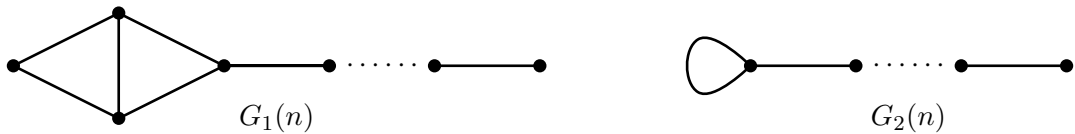
All the graphs we have encountered so far have had quite small nim-values.
In \cite{FS} the authors conjectured that the nim-values of graph games are unbounded. This was demonstrated in \cite{DT} with the family of graphs $G_1(n)$ on $n$ vertices in Figure \ref{unb12}.
Let
\begin{equation*}
  \lambda(k):=\left\{
                \begin{array}{ll}
                  2m, & \hbox{if $k=3m+0$;} \\
                  2m+1, & \hbox{if $k=3m+1$;} \\
                  2m, & \hbox{if $k=3m+2$.}
                \end{array}
              \right.
\end{equation*}
So the first few values of $\lambda$, starting with $\lambda(0)$, are $0,1,0,2,3,2,4,5,4,\dots$.
Then an induction argument, given in the proof of \cite[Thm. 3]{Ri}, shows that
\begin{equation*}
  \n(G_1(n)) = 2\cdot \lambda(n-3) \quad \text{for} \quad n\gqs 4.
\end{equation*}
The  family $G_2(n)$ on $n$ vertices in Figure  \ref{unb12} is just a path with a loop at the end. A similar proof shows
\begin{equation*}
  \n(G_2(n)) = 2\cdot \lambda(n) \quad \text{for} \quad n\gqs 0.
\end{equation*}

Some examples of the nim-values that arise when an odd cycle is attached to a tree are explored in Appendix A of \cite{Ri}. The next result generalizes Theorem 5 of \cite{Y} and demonstrates that the nim-values of ``$4$ or more''  from the statements of Theorems \ref{AB2} and \ref{AB} may become arbitrarily large.

\SpecialCoor
\psset{griddots=5,subgriddiv=0,gridlabels=0pt}
\psset{xunit=0.7cm, yunit=0.7cm, runit=0.7cm}
\psset{linewidth=1pt}
\psset{dotsize=5pt 0,dotstyle=*}
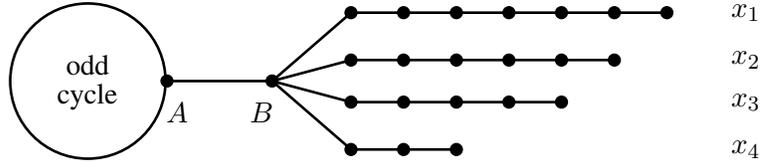
\begin{figure}[ht]
\centering
\begin{pspicture}(-1,-0.5)(12,3.4) 

\psset{arrowscale=2,arrowinset=0.5}

\pscircle(0.5,1.5){1.5}
\rput(0.5,1.8){odd}
\rput(0.5,1.2){cycle}
\rput(2.2,0.9){$A$}
\rput(3.8,0.9){$B$}

\rput(13,2.8){$x_1$}
\rput(13,1.9){$x_2$}
\rput(13,1.1){$x_3$}
\rput(13,0.2){$x_4$}

\psline(2,1.5)(4,1.5)(5.5,2.8)(11.5,2.8)
\psline(4,1.5)(5.5,1.9)(10.5,1.9)
\psline(4,1.5)(5.5,1.1)(9.5,1.1)
\psline(4,1.5)(5.5,0.2)(7.5,0.2)
\psdots(2,1.5)(4,1.5)
\multirput(5.5,2.8)(1,0){7}{\psdot(0,0)}
\multirput(5.5,1.9)(1,0){6}{\psdot(0,0)}
\multirput(5.5,1.1)(1,0){5}{\psdot(0,0)}
\multirput(5.5,0.2)(1,0){3}{\psdot(0,0)}

\end{pspicture}
\caption{An example of $R_{x_1,x_2, \dots,x_n}$ with $n=4$ and $(x_1,x_2,x_3,x_4)=(7,6,5,3)$}
\label{rg}
\end{figure}

Define the family of graphs  $R_{x_1,x_2, \dots,x_n}$ for positive integers $x_1, \dots, x_n$  as follows.  They consist of an edge $AB$ with an  odd cycle attached to $A$  and   $n$ paths of lengths $x_1, \dots, x_n$ attached to $B$ as in Figure \ref{rg}.

Set $\ell(x):= 2\cdot \lambda(x-1)$.  The first few values of $\ell$, starting with $\ell(1)$, are $0,2,0,4,6,4,8,10,8,\dots$.

\begin{theorem} \label{cat}
 We have
\begin{equation*}
  \n(R_{x_1,x_2, \dots,x_n})= \left\{
           \begin{array}{ll}
             \f(R_{x_1,x_2, \dots,x_n})  & \hbox{if $n$ is odd;} \\
            \ell(x_1)\oplus \cdots \oplus \ell(x_n)+4  & \hbox{if $n$ is even.}
           \end{array}
         \right.
\end{equation*}
\end{theorem}
\begin{proof}
  Set $L:=\ell(x_1)\oplus \cdots \oplus \ell(x_n)$ and the same nim sum with $\ell(x_r)$ removed is $L \oplus \ell(x_r)$. Let $X :=x_1+\cdots+x_n$. For $n=X=0$ we understand the empty nim-sum $L$ is $0$. To compute $\n(R_{x_1,x_2, \dots,x_n})$ we look at all possible moves and use induction on $X$.  The base case with $n=X=0$ has nim-value $4$ by Theorem \ref{AB}. When $n$ is odd the result also follows from Theorem \ref{AB}, so we may assume $n$ is even.

From the cycle and vertices $A$ and $B$ we obtain moves with nim-values $0,$ $1,$ $2$ and $3$. The only possible nim-values outside of these come from removing vertices and edges from the $n$ paths giving:
\begin{gather*}
  L \oplus \ell(x_r) \oplus \ell(x_r-1)+4, \\
(L \oplus \ell(x_r) \oplus \ell(x_r-1)+4)\oplus 1, \\
 (L \oplus \ell(x_r) \oplus \ell(i)+4)\oplus 1, \\
(L \oplus \ell(x_r) \oplus \ell(i)+4)\oplus 2
\end{gather*}
for all $r,$  $i$ satisfying $1\lqs r\lqs n$ and $1\lqs i \lqs x_r-2$. Noting that $(a+4)\oplus b = (a\oplus b)+4$ for $b=1,2$ we obtain
\begin{equation} \label{bx}
  \n(R_{x_1,x_2, \dots,x_n})=4+\mex\left(\Bigl\{  L \oplus \ell(x_r) \oplus \ell(x_r-1)\oplus (b-1), L \oplus \ell(x_r) \oplus \ell(i)\oplus b \Bigr\}\right)
\end{equation}
where $1\lqs r\lqs n$, $1\lqs i \lqs x_r-2$ and $b=1,2$. To compute this, let
\begin{equation*}
  E(k):=\Bigl\{ \ell(k-1), \ell(k-1)\oplus 1\Bigr\} \cup \Bigl\{\ell(i)\oplus 1, \ell(i)\oplus 2 \, \Big| \, 1\lqs i \lqs k-2 \Bigr\}.
\end{equation*}
It is straightforward to prove that, for all positive integers $k$,
\begin{equation} \label{lmex}
  \ell(k) = \mex(E(k)).
\end{equation}
We may rewrite \eqref{bx} as
\begin{equation} \label{tru}
  \n(R_{x_1,x_2, \dots,x_n})=4+\mex\left(\Bigl\{  L \oplus \ell(x_r) \oplus a \Bigr\}\right)
\end{equation}
where $1\lqs r\lqs n$ and $a \in E(x_r)$. It follows from \eqref{lmex}, \eqref{tru}  and the addition of games relation \eqref{addg} that $\n(R_{x_1,x_2, \dots,x_n})=4+L$ as desired.
\end{proof}

\SpecialCoor
\psset{griddots=5,subgriddiv=0,gridlabels=0pt}
\psset{xunit=0.7cm, yunit=0.7cm, runit=0.7cm}
\psset{linewidth=1pt}
\psset{dotsize=5pt 0,dotstyle=*}
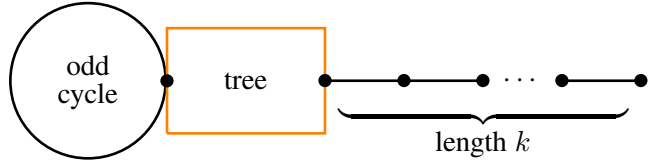
\begin{figure}[ht]
\centering
\begin{pspicture}(5,0.5)(17,4) 

\psset{arrowscale=2,arrowinset=0.5}

\pscircle(6.5,2.5){1.5}

\psline[linecolor=orange](8,1.5)(11,1.5)(11,3.5)(8,3.5)(8,1.5)
\psdots(8,2.5)

\rput(14,1.5){$\underbrace{\qquad \qquad \qquad \qquad \qquad}_{\text{\normalsize length $k$}}$}

\rput(9.5,2.5){tree}
\psline(11,2.5)(14,2.5)
\psline(15.5,2.5)(17,2.5)
\multirput(11,2.5)(1.5,0){5}{\psdot(0,0)}

\rput(14.75,2.5){$\cdots$}
\rput(6.5,2.8){odd}
\rput(6.5,2.2){cycle}

\end{pspicture}
\caption{Connecting an odd cycle, a tree and a path}
\label{conj1y}
\end{figure}
Suppose we fix each $x_i$ in $R_{x_1,x_2, \dots,x_n}$ except for $x_1$ say, and let $k=x_1$ increase. By Theorem \ref{cat}, the sequence of nim-values $g_k:=\n(R_{k,x_2, \dots,x_n})$   will be $0,3, 0,3,0,3,0,3, \dots$  if $n$ is odd and $\ell(k)\oplus r+4$, for some even number $r\gqs 0$, if $n$ is even (since the image of $\ell$ is the set of all even numbers). Khandhawit and  Ye investigated what  happens in general when a path of length $k$, a tree and an odd cycle are attached together as in Figure \ref{conj1y}.
They found two types of behavior for the sequence of nim-values $g_k$ for large $k$; see Table 6 of \cite{Y}. We see there is also a third type of behavior (which invalidates their Conjecture 1) given in (iii) next. In the examples we have computed, the sequence $g_1,g_2,g_3,\dots$ eventually matches one of the following three sequences:
\begin{enumerate}
\item  $\ell(1)\oplus r+4, \ell(2)\oplus r+4, \ell(3)\oplus r+4, \dots$ for some fixed $r\gqs 0$,
  \item the period $2$ sequence $4m,4m+3,4m,4m+3,\dots$ for some fixed $m$
  \item or the period $2$ sequence  $4m+1,4m+2,4m+1,4m+2,\dots$ for some fixed $m$.
\end{enumerate}
An instance of (iii) is shown for the family $H_{1,k}$ in Figure \ref{conjex}. The family $H_{2,k}$ in that figure shows that $r$ in (i) may be odd as the sequence is $g_k=\ell(k)\oplus 3 +4$ for $k\gqs 13$.
The examples in Figure \ref{conjex} also confirm Theorem \ref{AB2}; we see the highlighted vertices in $H_{1,k}$ and $H_{2,k}$ are odd degree telescoping vertices and the nim-values of these graphs are $4$ or more. By comparison, $H_{3,k}$ is similar to $H_{2,k}$ but does not have a telescoping vertex and its nim-values are $\f(H_{3,k})$.

\SpecialCoor
\psset{griddots=5,subgriddiv=0,gridlabels=0pt}
\psset{xunit=0.7cm, yunit=0.7cm, runit=0.7cm}
\psset{linewidth=1pt}
\psset{dotsize=5pt 0,dotstyle=*}
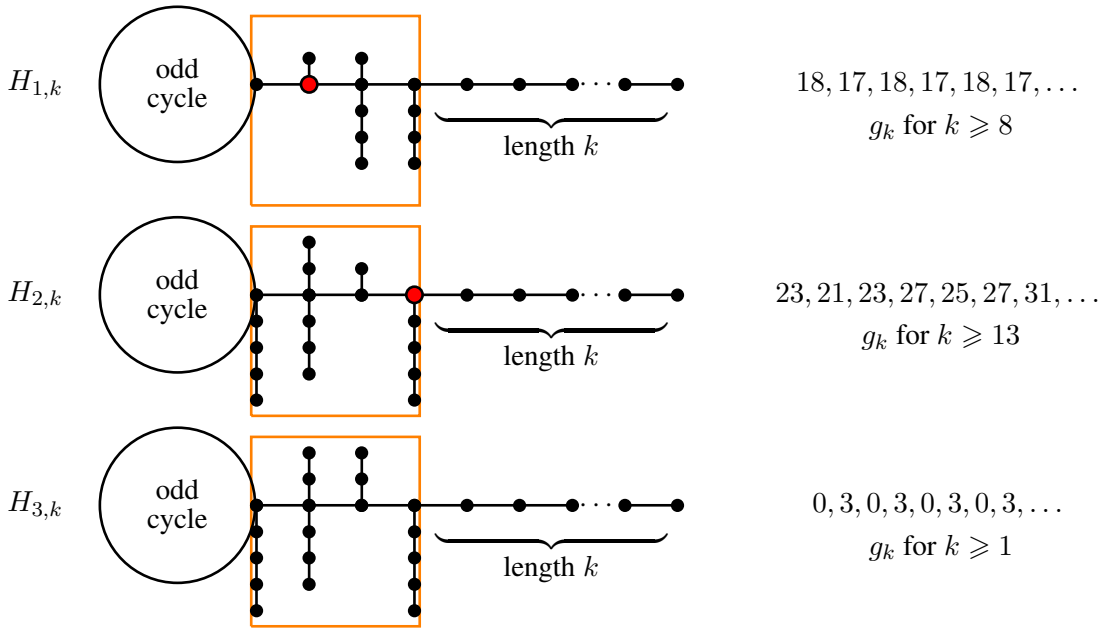
\begin{figure}[ht]
\centering
\begin{pspicture}(3,0)(24,12) 

\psset{arrowscale=2,arrowinset=0.5}

\psline[linecolor=orange](7.9,0.2)(11.1,0.2)(11.1,3.8)(7.9,3.8)(7.9,0.2)
\pscircle(6.5,2.5){1.5}

\rput(13.6,1.5){$\underbrace{\qquad \qquad \qquad \qquad}_{\text{\normalsize length $k$}}$}

\psline(8,2.5)(14,2.5)
\psline(15,2.5)(16,2.5)

\psline(8,2.5)(8,0.5)
\multirput(8,2.5)(0,-0.5){5}{\psdot(0,0)}
\psline(9,3.5)(9,1)
\multirput(9,3.5)(0,-0.5){6}{\psdot(0,0)}
\psline(10,2.5)(10,3.5)
\multirput(10,2.5)(0,0.5){3}{\psdot(0,0)}
\psline(11,2.5)(11,0.5)
\multirput(11,2.5)(0,-0.5){5}{\psdot(0,0)}

\multirput(8,2.5)(1,0){9}{\psdot(0,0)}

\rput(14.5,2.5){$\cdots$}
\rput(6.5,2.8){odd}
\rput(6.5,2.2){cycle}
\rput(3.8,2.5){$H_{3,k}$}

\rput(21,2.5){$0,3,0,3,0,3,0,3,\dots$}
\rput(21,1.7){$g_k$ for $k\gqs 1$}


\psline[linecolor=orange](7.9,4.2)(11.1,4.2)(11.1,7.8)(7.9,7.8)(7.9,4.2)
\pscircle(6.5,6.5){1.5}

\rput(13.6,5.5){$\underbrace{\qquad \qquad \qquad \qquad}_{\text{\normalsize length $k$}}$}

\psline(8,6.5)(14,6.5)
\psline(15,6.5)(16,6.5)

\psline(8,6.5)(8,4.5)
\multirput(8,6.5)(0,-0.5){5}{\psdot(0,0)}
\psline(9,7.5)(9,5)
\multirput(9,7.5)(0,-0.5){6}{\psdot(0,0)}
\psline(10,6.5)(10,7)
\multirput(10,6.5)(0,0.5){2}{\psdot(0,0)}
\psline(11,6.5)(11,4.5)
\multirput(11,6.5)(0,-0.5){5}{\psdot(0,0)}

\multirput(8,6.5)(1,0){9}{\psdot(0,0)}

\pscircle[fillstyle=solid,fillcolor=red](11,6.5){0.18}

\rput(14.5,6.5){$\cdots$}
\rput(6.5,6.8){odd}
\rput(6.5,6.2){cycle}
\rput(3.8,6.5){$H_{2,k}$}

\rput(21,6.5){$23, 21, 23, 27, 25, 27, 31,\dots$}
\rput(21,5.7){$g_k$ for $k\gqs 13$}


\psline[linecolor=orange](7.9,8.2)(11.1,8.2)(11.1,11.8)(7.9,11.8)(7.9,8.2)
\pscircle(6.5,10.5){1.5}

\rput(13.6,9.5){$\underbrace{\qquad \qquad \qquad \qquad}_{\text{\normalsize length $k$}}$}

\psline(8,10.5)(14,10.5)
\psline(15,10.5)(16,10.5)

\psline(9,10.5)(9,11)
\multirput(9,10.5)(0,0.5){2}{\psdot(0,0)}
\psline(10,11)(10,9)
\multirput(10,11)(0,-0.5){5}{\psdot(0,0)}
\psline(11,10.5)(11,9)
\multirput(11,10.5)(0,-0.5){4}{\psdot(0,0)}

\multirput(8,10.5)(1,0){9}{\psdot(0,0)}

\pscircle[fillstyle=solid,fillcolor=red](9,10.5){0.18}

\rput(14.5,10.5){$\cdots$}
\rput(6.5,10.8){odd}
\rput(6.5,10.2){cycle}
\rput(3.8,10.5){$H_{1,k}$}

\rput(21,10.5){$18,17,  18,17, 18,17, \dots$}
\rput(21,9.7){$g_k$ for $k\gqs 8$}

\end{pspicture}
\caption{Examples of nim-value sequences}
\label{conjex}
\end{figure}
Are any further types of sequences possible? In general we may ask what kinds of nim-value sequences arise when a path of length $k$ is attached to any graph.

\section{Nim-values of some graphs with many odd cycles or loops} \label{many}

Let $G$ be a connected graph without loops. If  every vertex of $G$ has degree at most two then it is just a path or a cycle. If we allow one vertex $P$ to have higher degree, then $G$ must consist of a number of cycles and paths attached to $P$.
\SpecialCoor
\psset{griddots=5,subgriddiv=0,gridlabels=0pt}
\psset{xunit=0.7cm, yunit=0.7cm, runit=0.7cm}
\psset{linewidth=1pt}
\psset{dotsize=5pt 0,dotstyle=*}
\begin{figure}[ht]
\centering
\begin{pspicture}(-1,-0.3)(12,3.4) 

\psset{arrowscale=2,arrowinset=0.5}

\pscircle(2.5,1.5){1.5}
\pscircle(2.2,1.5){1.8}
\pscircle(2.8,1.5){1.2}
\rput(-1.3,1.5){$r$ odd cycles}
\rput(4.2,0.7){$P$}

\rput(13,2.8){$x_1$}
\rput(13,1.9){$x_2$}
\rput(13,1.1){$x_3$}
\rput(13,0.2){$x_4$}

\psline(4,1.5)(5.5,2.8)(11.5,2.8)
\psline(4,1.5)(5.5,1.9)(10.5,1.9)
\psline(4,1.5)(5.5,1.1)(9.5,1.1)
\psline(4,1.5)(5.5,0.2)(7.5,0.2)
\psdot(4,1.5)
\multirput(5.5,2.8)(1,0){7}{\psdot(0,0)}
\multirput(5.5,1.9)(1,0){6}{\psdot(0,0)}
\multirput(5.5,1.1)(1,0){5}{\psdot(0,0)}
\multirput(5.5,0.2)(1,0){3}{\psdot(0,0)}

\end{pspicture}
\caption{A graph from Theorem \ref{trro}}
\label{rodd}
\end{figure}
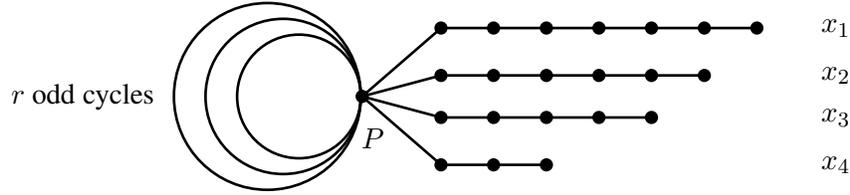
Attaching an even cycle to a vertex in a graph is the same as adding a disjoint vertex, by Lemma \ref{twin}, and the nim-value simply changes by $\oplus 1$. So we may assume that the $r$  cycles attached to $P$ are all odd. Let the $m$ paths attached to $P$ have lengths $x_1 \gqs x_2 \gqs \cdots \gqs x_m \gqs 0$ and it is convenient to set
\begin{equation}\label{xsa}
X:=\sum_{i=1}^m x_i, \qquad \hat{X}:= \sum_{i=1}^m (-1)^{i+1} x_i.
\end{equation}
See for example Figure \ref{rodd}.
We have $\hat{X}=0$ if and only if the paths cancel in pairs and $G$ reduces to just the $r$ cycles. Similarly, $\hat{X}=1$ if and only if  $G$ reduces to the $r$ cycles with two paths of lengths differing by $1$ (or one path of length $1$) attached at $P$.

\begin{theorem} \label{trro}
Let $G$ consist of a vertex $P$ to which $r$ odd cycles and $m$ paths are attached. Then with the above notation we have
\begin{equation*}
  \n(G)= \left\{
           \begin{array}{ll}
             0 & \hbox{for $r$ odd and $\hat{X}=0$;} \\
             4 & \hbox{for $r$ odd and $\hat{X}=1$;} \\
             \f(G) & \hbox{otherwise.}
           \end{array}
         \right.
\end{equation*}
\end{theorem}
\begin{proof}
We use induction on $(r,X)$ ordered lexicographically. In other words, $(r,X)>(r',X')$ exactly when $r>r'$ or when $r=r'$ and $X>X'$. Any move of $G$ gives a graph with smaller $(r,X)$ (and removing $P$ gives a bipartite graph). The base case of the induction is true since $r=X=0$ means $G$ is a single vertex.

Note that
\begin{equation*}
  \f(G)=(X+1)_{(2)}+2((X+r)_{(2)}).
\end{equation*}
The following cases establish the  argument.
\begin{description}
  \item[{\it Case (i).}] Suppose $r$ is odd and $\hat{X}=0$. Then $\deg P$ is even, $X$ is even and $\f(G)=3$. We want to show that $\n(G)=0$ and this follows if  all moves $H$ have $\n(H)\neq 0$. Removing $P$ or removing an edge or vertex on a cycle gives $H$ with $\n(H)=\f(H)$. But for these moves $\f(H)\neq 0$ since we may only get $0$ by removing an odd degree vertex. If $H$ is a move that deletes a vertex or edge from one of the paths of $G$ then clearly $\n(H) \neq 0$ if  $\hat{X}\to 1$. The last option is that $H$ makes $\hat{X}$  greater than $1$ so that $\n(H)=\f(H)$. If $\f(H)= 0$ then $H$ removes a vertex of odd degree. However, the only path move that does this removes a leaf vertex and has  $\hat{X}\to 1$. Hence $\n(H)\neq 0$ in this option.

  \item[{\it Case (ii).}]  Suppose $r$ is odd and $\hat{X}=1$. Then  $X$ is odd, $\f(G)=0$ and we want to show that $\n(G)=4$.
First we prove that moves with nim-values $0,$ $1,$ $2$ and $3$ exist. Deleting a vertex or edge on the cycle gives  nim-values $1$ and $2$. Since $\hat{X}=1$, there exists $r$ such that $x_r=1+x_{r+1}$ (with $x_{r+1}$ possibly $0$). Removing the vertex at the end of the path of length $x_r$ makes $\hat{X}\to 0$, giving nim-value $0$. If $\deg P$ is odd then removing $P$ is a move with nim-value $3$. Otherwise $\deg P$ is even implying there exists $r$  such that $x_r=1+x_{r+1}$ and $x_{r+1} \gqs 1$. Removing the vertex at the end of the path of length $x_{r+1}$ makes $\hat{X} \to 2$ and this move has nim-value $3$.

Next we show that all moves $H$ have $\n(H)\neq 4$. The only possible moves with nim-value $\gqs 4$ have $\hat{X}$ remaining as $1$. If $x_r=1+x_{r+1}$ then the only move that does this has $x_r \to x_r-2$ by removing a degree $2$ vertex. But this move has nim-value $4\oplus 1 =5$.

  \item[{\it Case (iii).}] Suppose $r$ is even and $X$ is even.  Then   $\f(G)=1$ and we want to show that $\n(G)=1$. To find a move with nim-value $0$ we may remove a degree $2$ vertex on one of the paths of $G$. We cannot do this if $x_1 \lqs 1$. Since $X$ is even, it follows that the degree of $P$ must be even if $x_1 \lqs 1$. Removing $P$ then gives the  move with nim-value $0$.
 If $H$ is any move then  $\n(H)=0$, $4$ or $\f(H)$ and not equal to $1$.

\item[{\it Case (iv).}] Suppose $r$  is even and $X$ is odd.  Then   $\f(G)=2$ and we want to show that $\n(G)=2$. We have $x_1\gqs 1$ and removing the end edge and vertex on this path gives nim-values $0$ and $1$ respectively. If $H$ is any move then $\n(H)$ is $0$, $4$ or $\f(H)$ and  not equal to $2$.

\item[{\it Case (v).}] Suppose $r$ is odd, $X$ is even and $\hat{X}\gqs 2$. Then   $\f(G)=3$ and we want to show that $\n(G)=3$. Removing a cycle edge and vertex gives nim-values $1$ and $2$ respectively. If $\deg P$ is odd the removing it gives nim-value $0$. Otherwise, the largest $r$ for which $x_r$ is positive is even. Removing the vertex at the end of this path increases $\hat{X}$ and therefore this move has nim-value $0$. We have shown that moves with nim-values $0,$ $1$  and $2$ exist.

    It remains to show that all moves $H$ have $\n(H)\neq 3$. By induction we have $\n(H)=0\oplus t$ or $4\oplus t$ or $\f(H)$ for $t=0,$ $1$ or $2$, the nim-value of the disconnected path. It follows that $\n(H)\neq 3$.

\item[{\it Case (vi).}] Suppose $r$ is odd, $X$ is odd and $\hat{X}\gqs 2$. Then   $\f(G)=0$ and we want to show that $\n(G)=0$. This is true if all moves $H$ have $\n(H)\neq 0$. As in the previous case, $\n(H)=0\oplus t$ or $4\oplus t$ or $\f(H)$ for $t=0,$ $1$ or $2$. The only way to obtain $\n(H)=0$ is if a vertex or edge is removed so that $\hat{X} \to 0$ and the disconnected path has nim-value $t=0$. This is not possible for $\hat{X}\gqs 2$. \qedhere
\end{description}
\end{proof}

\SpecialCoor
\psset{griddots=5,subgriddiv=0,gridlabels=0pt}
\psset{xunit=0.5cm, yunit=0.5cm, runit=0.5cm}
\psset{linewidth=1pt}
\psset{dotsize=5pt 0,dotstyle=*}

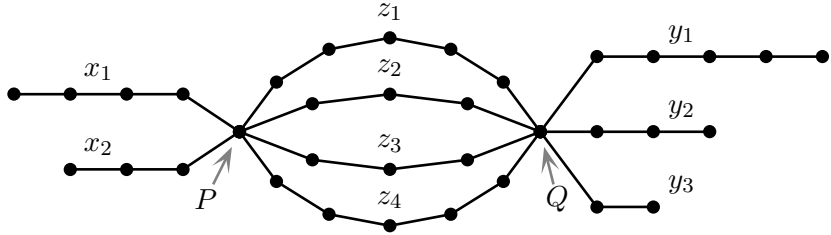
\begin{figure}[ht]
\centering
\begin{pspicture}(6,2)(26,8) 

\psset{arrowscale=2,arrowinset=0.5}

\savedata{\mydataz}[
{{20., 5.00005}, {19.0163, 6.32177}, {17.6191, 7.19499}, {16.,
  7.5}, {14.3809, 7.19499}, {12.9837, 6.32177}, {12., 5.00005}}
]
\dataplot[plotstyle=line]{\mydataz}
\dataplot[plotstyle=dots]{\mydataz}

\savedata{\mydataza}[
{{19.9997, 5.00017}, {18.0614, 5.74626}, {16., 6.}, {13.9386,
  5.74626}, {12.0003, 5.00017}}
]
\dataplot[plotstyle=line]{\mydataza}
\dataplot[plotstyle=dots]{\mydataza}

\savedata{\mydatazb}[
{{19.9997, 4.99983}, {18.0614, 4.25374}, {16., 4.}, {13.9386,
  4.25374}, {12.0003, 4.99983}}
]
\dataplot[plotstyle=line]{\mydatazb}
\dataplot[plotstyle=dots]{\mydatazb}

\savedata{\mydatazc}[
{{20., 4.99995}, {19.0163, 3.67823}, {17.6191, 2.80501}, {16.,
  2.5}, {14.3809, 2.80501}, {12.9837, 3.67823}, {12., 4.99995}}
]
\dataplot[plotstyle=line]{\mydatazc}
\dataplot[plotstyle=dots]{\mydatazc}

\rput(16,3.2){$z_4$}
\rput(16,4.7){$z_3$}
\rput(16,6.7){$z_2$}
\rput(16,8.2){$z_1$}

\psline(6,6)(10.5,6)(12,5)(10.5,4)(7.5,4)
\multirput(6,6)(1.5,0){4}{\psdot(0,0)}
\multirput(7.5,4)(1.5,0){3}{\psdot(0,0)}

\psline(27.5,7)(21.5,7)(20,5)(21.5,3)(23,3)
\psline(20,5)(24.5,5)
\multirput(21.5,7)(1.5,0){5}{\psdot(0,0)}
\multirput(21.5,5)(1.5,0){3}{\psdot(0,0)}
\multirput(21.5,3)(1.5,0){2}{\psdot(0,0)}

\rput(11.1,3.2){$P$}
\psline[linecolor=gray]{->}(11.3,3.6)(11.8,4.6)
\rput(20.45,3.2){$Q$}
\psline[linecolor=gray]{->}(20.35,3.6)(20.1,4.6)

\rput(8.25,6.6){$x_1$}
\rput(8.25,4.6){$x_2$}

\rput(23.75,7.6){$y_1$}
\rput(23.75,5.6){$y_2$}
\rput(23.75,3.6){$y_3$}

\end{pspicture}
\caption{A graph from Theorem \ref{thmxyz}}
\label{xyz}
\end{figure}

We next consider a family of graphs where two vertices $P$ and $Q$ may have high degree and the remaining vertices have degree $\lqs 2$. Suppose there are $k$ paths of lengths $z_1, \dots, z_k \gqs 1$ linking $P$ and $Q$. We also have $m$ paths from $P$ of lengths $x_1 \gqs x_2 \gqs \cdots \gqs x_m \gqs 0$ and $n$ paths from $Q$ of lengths $y_1 \gqs y_2 \gqs \cdots \gqs y_n \gqs 0$ as shown in Figure \ref{xyz}. Similarly to \eqref{xsa}, put
\begin{equation*}
X:=\sum_{i=1}^m x_i, \qquad \hat{X}:= \sum_{i=1}^m (-1)^{i+1} x_i, \qquad Y:=\sum_{i=1}^n y_i, \qquad \hat{Y}:= \sum_{i=1}^n (-1)^{i+1} y_i, \qquad Z:=\sum_{i=1}^k z_i.
\end{equation*}

\begin{theorem} \label{thmxyz}
Let $G$ be a member of the above family of graphs involving paths linking to $P$ and $Q$. With the defined notation we have
\begin{equation*}
  \n(G)= \left\{
           \begin{array}{ll}
             0 & \hbox{for $k$ even, $Z$ odd and $\hat{X}+\hat{Y}=0$;} \\
             4 & \hbox{for $k$ even, $Z$ odd and $\hat{X}+\hat{Y}=1$;} \\
             \f(G) & \hbox{otherwise.}
           \end{array}
         \right.
\end{equation*}
\end{theorem}
\begin{proof}
We argue similarly to the proof of Theorem \ref{trro} and use induction on $(k,X+Y)$ ordered lexicographically.  The result is true in the base cases of $k=0$ (so that $Z=0$) and $k=1$ since $G$ is then bipartite. Hence we assume $k\gqs 2$. If $z_1=\cdots=z_k=1$ then we have a multiple edge which simplifies to a single edge or no edge as discussed after Lemma \ref{twin}. Therefore we may assume there exists a $z_i$ with $z_i\gqs 2$.
Also note that
\begin{equation*}
  \f(G)=(X+Y+Z+k)_{(2)}+2((X+Y+Z)_{(2)}).
\end{equation*}
The following cases establish the  argument.
\begin{description}
  \item[{\it Case (i).}] Suppose $k$ is even, $Z$ is odd and $\hat{X}+\hat{Y}=0$. Then $X,$ $Y,$ $\deg P$ and $\deg Q$ are all even and $\f(G)=3$. To show that $\n(G)=0$ we need to prove that all moves $H$ have $\n(H)\neq 0$. Removing $P$, $Q$ or  an edge or vertex between $P$ and $Q$ gives $H$ with $\n(H)=\f(H)$. But for these moves $\f(H)\neq 0$ since we may only get $0$ by removing an odd degree vertex. If $H$ is a move that deletes a vertex or edge from a path of $G$ not  between $P$ and $Q$ then clearly $\n(H) \neq 0$ if  $\hat{X}+\hat{Y}\to 1$. The last option is that $H$ makes $\hat{X}+\hat{Y}$  greater than $1$ so that $\n(H)=\f(H)$. If $\f(H)= 0$ then $H$ removes a vertex of odd degree. However, the only path move that does this removes a leaf vertex and has  $\hat{X}+\hat{Y}\to 1$. Hence $\n(H)\neq 0$ in this option.

  \item[{\it Case (ii).}]  Suppose $k$ is even, $Z$ is odd and $\hat{X}+\hat{Y}=1$. Then  $X+Y$ is odd, $\f(G)=0$ and we want to show that $\n(G)=4$.
First we prove that moves with nim-values $0,$ $1,$ $2$ and $3$ exist. Deleting a vertex or edge between $P$ and $Q$ gives  nim-values $1$ and $2$. With $\hat{X}+\hat{Y}=1$ we must have $\hat{X}=1$ or $\hat{Y}=1$. If $\hat{X}=1$, there exists $r$ such that $x_r=1+x_{r+1}$  (with $x_{r+1}$ possibly $0$). Removing the vertex at the end of the path of length $x_r$ makes $\hat{X}\to 0$, giving nim-value $0$. If $\deg P$ is odd then removing $P$ is a move with nim-value $3$. Otherwise $\deg P$ is even implying there exists $r$  such that $x_r=1+x_{r+1}$ and $x_{r+1} \gqs 1$. Removing the vertex at the end of the path of length $x_{r+1}$ makes $\hat{X} \to 2$ and this move has nim-value $3$. The same argument works if $\hat{Y}=1$.

Next we show that all moves $H$ have $\n(H)\neq 4$. The only possible moves with nim-value $\gqs 4$ have $\hat{X}+\hat{Y}$ remaining as $1$. If $\hat{X}=1$ and $x_r=1+x_{r+1}$ then the only move that does this has $x_r \to x_r-2$ by removing a degree $2$ vertex. But this move has nim-value $4\oplus 1 =5$. We have the same argument when $\hat{Y}=1$.

  \item[{\it Case (iii).}] Suppose $k$ is odd and $X+Y+Z$ is even.  Then   $\f(G)=1$ and we want to show that $\n(G)=1$. To find a move with nim-value $0$ we  may remove a degree $2$ vertex between $P$ and $Q$ if $Z$ is even. Now assume $Z$ is odd. Removing a degree $2$ vertex on one of the paths not between $P$ and $Q$ gives nim-value $0$. We cannot do this if all $x_i, y_i$ are $\lqs 1$. Since one of $X$ or $Y$ is odd, it follows that the degree of $P$ or $Q$ must be even if $x_i, y_i$ are $\lqs 1$. Removing this even degree vertex then gives the  move with nim-value $0$.
  If $H$ is any move then  $\n(H)=0$, $4$ or $\f(H)$ and not equal to $1$.

\item[{\it Case (iv).}] Suppose $k$ is odd and $X+Y+Z$ is odd.  Then   $\f(G)=2$ and we want to show that $\n(G)=2$. If any of $x_i, y_i$ are $\gqs 1$ then removing the end edge and vertex on this path gives nim-values $0$ and $1$ respectively. Otherwise, $X=Y=0$ and $Z,$ $\deg P$ are odd. Removing $P$ gives nim-value $1$ and removing a central edge from a path with $z_i$ odd gives nim-value $0$. If $H$ is any move then $\n(H)$ is $0$, $4$ or $\f(H)$ and  not equal to $2$.

\item[{\it Case (v).}] Suppose $k$ is even,  $X+Y+Z$ is odd and $\hat{X}+\hat{Y}\gqs 2$. Then   $\f(G)=3$ and we want to show that $\n(G)=3$. Removing a cycle edge and vertex gives nim-values $1$ and $2$ respectively. If $\deg P$ is odd the removing it gives nim-value $0$. Otherwise, the largest $r$ for which $x_r$ is positive is even. Removing the vertex at the end of this path increases $\hat{X}$ and therefore this move has nim-value $0$. We have shown that moves with nim-values $0,$ $1$  and $2$ exist.

    It remains to show that all moves $H$ have $\n(H)\neq 3$. By induction we have $\n(H)=0\oplus t$ or $4\oplus t$ or $\f(H)$ for $t=0,$ $1$ or $2$, the nim-value of the disconnected path. It follows that $\n(H)\neq 3$.

\item[{\it Case (vi).}] Suppose $k$ is even,  $X+Y+Z$ is even and $\hat{X}+\hat{Y}\gqs 2$.  Then   $\f(G)=0$ and we want to show that $\n(G)=0$. This is true if all moves $H$ have $\n(H)\neq 0$. As in the previous case, $\n(H)=0\oplus t$ or $4\oplus t$ or $\f(H)$ for $t=0,$ $1$ or $2$. The only way to obtain $\n(H)=0$ is if a vertex or edge is removed so that $\hat{X}+\hat{Y} \to 0$ and the disconnected path has nim-value $t=0$. This is not possible for $\hat{X}+\hat{Y}\gqs 2$. \qedhere
\end{description}
\end{proof}

Propositions \ref{oddcyc} and \ref{paths}  follow as special cases of Theorems \ref{trro} and \ref{thmxyz} respectively.
Note that $\hat{X}=1$ in Theorem \ref{trro} and $\hat{X}+\hat{Y}=1$ in Theorem \ref{thmxyz} exactly when there is a single telescoping vertex of odd degree. So these theorems fit a similar pattern to the results in Section \ref{1cyc} and could be part of a larger encompassing theory.

 We consider one more family of graphs in this section.

\begin{theorem} \label{kngen}
Let $K_n(m)$ be the complete graph $K_n$ with a loop attached to $m$ different vertices. Then
\begin{equation*}
  \n(K_n(m)) = (m+n)_{(3)}.
\end{equation*}
\end{theorem}
\begin{proof}
 We use induction on $n$ with the $n=0,$  $1$ cases  easily verified. Assume $n\gqs 2$. If $m=0$ then we just have the complete graph and the theorem follows by \eqref{kn}. If $0 < m < n$ then $K_n(m)$ contains two vertices $v$ and $v'$ connected by an edge such that exactly one of them has a single loop attached.

To proceed we need the following extension of  the symmetry lemma. Suppose that $\tau:G\to G$ satisfies the conditions of Lemma \ref{twin} with $u\neq \tau(u)$ for vertex $u$. Let $G^*$ be $G$ with an edge $e$ added between $u$ and $\tau(u)$ and a loop $l$ added to either vertex. The proof of Lemma \ref{twin} goes through if we respond to $e$ with $l$ and vice versa. This proves $\n(G^*)=\n(G^\tau)$.

Applying the above argument, where $\tau$ maps $v \to v'$ and fixes the remaining vertices, shows that
$$
\n(K_n(m))=\n(K_{n-2}(m-1))=(m+n-3)_{(3)}=(m+n)_{(3)}
$$
by induction. In the final case, $m=n$ and all vertices of $K_n(n)$ have a loop attached. The three possible moves from this position are to:
\begin{enumerate}
  \item  remove a loop and get a graph with nim-value equal to $\n(K_{n-2}(n-2))$ by the above extension to the symmetry lemma,
  \item  remove an edge and get a graph with nim-value equal to $\n(K_{n-2}(n-2))$ by the symmetry lemma,
  \item  remove a vertex and get $K_{n-1}(n-1)$.
\end{enumerate}
Therefore
\begin{align*}
  \n(K_n(n)) & = \mex\bigl(\{\n(K_{n-2}(n-2)),\n(K_{n-1}(n-1))\}\bigr) \\
   & = \mex\bigl(\{(2n-4)_{(3)},(2n-2)_{(3)}\}\bigr)\\
 & = \mex\bigl(\{(2n+1)_{(3)},(2n+2)_{(3)}\}\bigr) = (2n)_{(3)}. \qedhere
\end{align*}
\end{proof}

\section{Wheel graphs and subgraphs} \label{wh}

As we saw in the introduction,  the  wheel graph $W_n$
is constructed by joining a central hub vertex to each vertex of the cycle graph $C_n$.  We will later need the {\em fan graph} $F_n$ and also $F^*_n$ which we may call a {\em fan with a handle}. Construct $F_n$ by removing a rim edge of $W_n$ and construct $F^*_n$  by removing two adjacent rim edges of $W_n$.
Examples are shown in Figure \ref{fans}.

\SpecialCoor
\psset{griddots=5,subgriddiv=0,gridlabels=0pt}
\psset{xunit=0.5cm, yunit=0.5cm, runit=0.5cm}
\psset{linewidth=1pt}
\psset{dotsize=5pt 0,dotstyle=*}

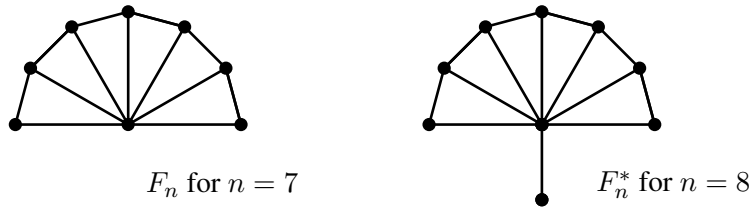
\begin{figure}[ht]
\centering
\begin{pspicture}(0,2.5)(21,8.5) 

\psset{arrowscale=2,arrowinset=0.5}


\savedata{\mydata}[
{{8., 5.}, {7.59808, 6.5}, {6.5, 7.59808}, {5., 8.}, {3.5,
  7.59808}, {2.40192, 6.5}, {2., 5.}}
]
\dataplot[plotstyle=line]{\mydata}
\dataplot[plotstyle=dots]{\mydata}
\psdots(5,5)(16,3)

\savedata{\mydataww}[
{{5,5},{8., 5.}, {7.59808, 6.5}, {5,5},{6.5, 7.59808}, {5., 8.}, {5,5},{3.5,
  7.59808}, {2.40192, 6.5}, {5,5},{2., 5.}}
]
\dataplot[plotstyle=line]{\mydataww}

\savedata{\mydataza}[
{{19., 5.}, {18.5981, 6.5}, {17.5, 7.59808}, {16., 8.}, {14.5,
  7.59808}, {13.4019, 6.5}, {13., 5.}}
]
\dataplot[plotstyle=line]{\mydataza}

\savedata{\mydataz}[
{{16,3},{16,5}, {19., 5.}, {18.5981, 6.5}, {16,5}, {17.5, 7.59808}, {16., 8.}, {16,5}, {14.5,
  7.59808}, {13.4019, 6.5}, {16,5}, {13., 5.}}
]
\dataplot[plotstyle=line]{\mydataz}
\dataplot[plotstyle=dots]{\mydataz}

\rput(7.5,3.4){$F_n$ for $n=7$}
\rput(19.5,3.4){$F^*_n$ for $n=8$}

\end{pspicture}
\caption{Some wheel subgraphs -- fans denoted $F_n$  and fans with a handle  denoted $F^*_n$}
\label{fans}
\end{figure}

The wheel graph $W_n$ for $n$ even is easily seen, with the symmetry lemma, to have nim-value $1$. This follows by letting $\tau$ fix a diameter and reflecting the graph from one side of the diameter to the other. The fixed diameter is a path of length $2$ with nim-value $1$. Alternatively, $\tau$ could fix the central hub  of $W_n$ and send each vertex and edge  to the opposite side. The nim-value of a single vertex is again $1$.
These techniques do not work for $W_n$ with $n$ odd since a short argument shows that any $\tau$ satisfying the conditions of Lemma \ref{twin} must be the identity  on $W_n$.

In a computer calculation we have found the nim-values of $W_n$ and all its subgraphs for $n\lqs 14$.
This proves directly that $\n(W_n)=1$ for $3\lqs n\lqs 14$. It also reveals interesting patterns that we describe throughout this section.

To prove $\n(W_n)=1$ for  $n>14$, we consider the graph made up of $W_n$ and an isolated vertex and want to show that the second player has a winning strategy. Clearly, removing the central hub vertex of $W_n$ has the reply of removing the isolated  vertex and vice versa. Also deleting a vertex on the rim of the wheel has the response of deleting the opposite edge and vice versa; the remaining graph is equivalent to two isolated vertices by symmetry.

\SpecialCoor
\psset{griddots=5,subgriddiv=0,gridlabels=0pt}
\psset{xunit=0.5cm, yunit=0.5cm, runit=0.5cm}
\psset{linewidth=1pt}
\psset{dotsize=5pt 0,dotstyle=*}

\begin{figure}[ht]
\centering
\begin{pspicture}(0,-0.4)(30,9) 

\psset{arrowscale=2,arrowinset=0.5}


\psline[linestyle=dashed,linecolor=red](5,5)(6.92836, 2.70187)
\psline[linestyle=dashed,linecolor=red](5,5)(7.95442, 4.47906)
\psline[linestyle=dashed,linecolor=red](5,5)(7.59808, 6.5)
\psline[linestyle=dashed,linecolor=red](5,5)(3.97394, 7.81908)
\psline[linestyle=dashed,linecolor=red](5,5)(2.40192, 6.5)
\psline[linestyle=dashed,linecolor=red](5,5)(2.04558, 4.47906)
\psline(5,2)(5,5)(3.07164, 2.70187)
\savedata{\mydata}[
{{5., 2.}, {6.92836, 2.70187}, {7.95442, 4.47906}, {7.59808,
  6.5}, {6.02606, 7.81908}, {3.97394, 7.81908}, {2.40192,
  6.5}, {2.04558, 4.47906}, {3.07164, 2.70187}, {5., 2.}}
]
\dataplot[plotstyle=line]{\mydata}
\dataplot[plotstyle=dots]{\mydata}
\psdot(5,5)
\pscircle[fillstyle=solid,fillcolor=red](5,2){0.22}
\pscircle[fillstyle=solid,fillcolor=red](3.07164, 2.70187){0.22}

\savedata{\mydataww}[
{{28.6219, 2.47624}, {27,5}, {29.7289, 3.75375}, {29.9695, 5.42694},{27,5},  {29.2672,
   6.96458}, {27.8452, 7.87848}, {26.1548, 7.87848}, {27,5}, {24.7328,
  6.96458}, {24.0305, 5.42694},{27,5},  {24.2711, 3.75375}, {25.3781,
  2.47624},{27,5} }
]
\dataplot[plotstyle=line]{\mydataww}

\savedata{\mydataw}[
{{28.6219, 2.47624}, {29.7289, 3.75375}, {29.9695, 5.42694}, {29.2672,
   6.96458}, {27.8452, 7.87848}, {26.1548, 7.87848}, {24.7328,
  6.96458}, {24.0305, 5.42694}, {24.2711, 3.75375}, {25.3781,
  2.47624}}
]
\dataplot[plotstyle=line]{\mydataw}
\psline[linecolor=white, linewidth=2pt](29.9695, 5.42694)(29.2672,  6.96458)
\psline[linestyle=dashed](29.9695, 5.42694)(29.2672,  6.96458)
\psline[linecolor=white, linewidth=2pt](27.8452, 7.87848)(26.1548, 7.87848)
\psline[linestyle=dashed](27.8452, 7.87848)(26.1548, 7.87848)
\dataplot[plotstyle=dots]{\mydataw}
\psdot(27,5)

\rput(28.8, 1.7){$v(4k+2)$}
\rput(25.1, 1.7){$v(1)$}
\rput(29, 8.5){$v(2k+2)$}
\rput(30.2, 6.3){$e_r$}
\rput(27, 7.4){$e_l$}
\rput(27, 4.2){$c$}

\savedata{\mydataza}[
{{17.9284, 2.70187}, {16,5}, {18.9544, 4.47906}, {18.5981, 6.5},  {16,5},
   {14.9739, 7.81908}, {13.4019, 6.5}, {16,5},  {13.0456,
  4.47906}, {14.0716, 2.70187} {16,5}}
]
\dataplot[plotstyle=line]{\mydataza}

\savedata{\mydataz}[
{{17.9284, 2.70187}, {18.9544, 4.47906}, {18.5981, 6.5}, {17.0261,
  7.81908}, {14.9739, 7.81908}, {13.4019, 6.5}, {13.0456,
  4.47906}, {14.0716, 2.70187}}
]
\dataplot[plotstyle=line]{\mydataz}
\psline[linecolor=white, linewidth=2pt](18.5981, 6.5)(17.0261, 7.81908)
\psline[linestyle=dashed](18.5981, 6.5)(17.0261, 7.81908)
\psline[linecolor=white, linewidth=2pt](14.9739, 7.81908)(13.4019, 6.5)
\psline[linestyle=dashed](14.9739, 7.81908)(13.4019, 6.5)
\dataplot[plotstyle=dots]{\mydataz}

\rput(17.8, 1.8){$v(4k)$}
\rput(14.1, 1.8){$v(1)$}
\rput(18, 8.5){$v(2k+1)$}
\rput(18.2, 7.5){$e_r$}
\rput(13.8, 7.5){$e_l$}
\rput(16, 4.2){$c$}

\psdots(0.5,5)(11.5,5)(22.5,5)(16,5)

\rput(4.25,0.5){$K_1 \cup W_9$ less a spoke}
\rput(15.25,0.5){$Q_{n}$ for  $n=4k+1=9$}
\rput(26.25,0.5){$Q_{n}$ for  $n=4k+3=11$}

\end{pspicture}
\caption{Wheel subgraphs with labelling}
\label{1315}
\end{figure}
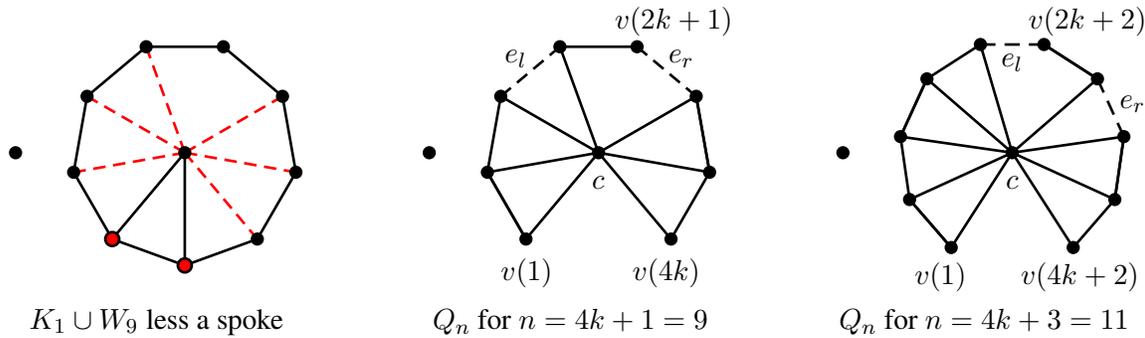

If the first player removes a spoke then the winning responses in $K_1 \cup W_n$, with $n=9$ for example, are highlighted on the left in Figure \ref{1315}: six spokes and two vertices. This same pattern appears for all odd $n \lqs 13$. Removing one of the indicated vertices  leaves the graph labelled $Q_{9}$ in the middle of  Figure \ref{1315}. In general we let $Q_n$ be the graph consisting of $K_1$ and $W_n$  with one spoke and one of the two opposite vertices deleted. Label the vertices of $Q_n$ as shown in Figure \ref{1315}, $v(1), \cdots ,v(n-1)$ with central hub vertex $c$. The missing spoke is between $c$ and $v((n+1)/2)$.

\begin{theorem} \label{wheel}
We have $\n(W_n)=1$ for all $n$ in the range $3\lqs n\lqs 25$.
\end{theorem}
\begin{proof}
We already saw that $\n(W_n)=1$ for $n$ even and that $\n(W_n)=1$ for $3\lqs n\lqs 14$. This computation also shows that $\n(Q_n)=0$ for $n$ odd in the range $3\lqs n\lqs 13$. We prove the theorem by demonstrating that $\n(Q_n)=0$ for all odd $n$ in the range $15\lqs n\lqs 25$.

Suppose that $\n(Q_m)=0$ for all odd $m$ in the range $3\lqs m\lqs n-2$ and consider $Q_n$. We first look at the case $n=4k+1$ as shown in the middle of  Figure \ref{1315}. Let $e_l$ be the edge between $v(2k-1)$ and $v(2k)$. Let $e_r$ be the edge between $v(2k+1)$ and $v(2k+2)$. To show that $\n(Q_n)=0$ we look for winning responses to any moves of the first player. If $K_1$ is removed then the winning response is $v(1)$ since an application of the symmetry lemma shows the remaining graph has the same nim-value as two isolated vertices. Also deleting $v(1)$ is a winning response to deleting $K_1$. We may write this move/response pair as $K_1 \leftrightarrow v(1)$. Exercises with the symmetry lemma give the following pairs:  $c \leftrightarrow v(1),$ $v(2k) \leftrightarrow v(2k+1)$ and, since $\n(Q_m)=0$ for smaller $m$, $v(i) \leftrightarrow v(n-i)$ for $1\lqs i \lqs n-1$. For edge moves we have
$(v(2k),v(2k+1)) \leftrightarrow (c,v(1))$ and also  $(v(i),v(i+1)) \leftrightarrow (v(n-i-1),v(n-i))$  for $1\lqs i \lqs 2k-1$.

It remains to find winning replies to the first player removing a spoke. We choose the response of removing $e_l$ when any of the spokes on the left are removed: $(v(i),c)$ with  $1\lqs i \lqs 2k$. The resulting graph simplifies to $F^*_{2k+2}$ with two spokes missing (the edge connected to the degree one vertex must remain).  We choose the response of removing $e_r$ when any of the spokes on the right are removed: $(v(i),c)$ with  $2k+2\lqs i \lqs 4k$. The resulting graph also simplifies to $F^*_{2k+2}$ with two spokes missing.

For the case $n=4k+3$ we argue similarly, with the initial moves and responses using the same symmetries.  The spoke move responses are defined slightly differently with
$e_l$  the edge between $v(2k+1)$ and $v(2k+2)$, and  $e_r$  the edge between $v(2k+3)$ and $v(2k+4)$. We choose the response of removing $e_l$ when any of the spokes on the left are removed: $(v(i),c)$ with  $1\lqs i \lqs 2k+1$. The resulting graph simplifies to $F^*_{2k+2}$ with two spokes missing.  We choose the response of removing $e_r$ when any of the spokes on the right are removed: $(v(i),c)$ with  $2k+3\lqs i \lqs 4k+2$. The resulting graph simplifies to $F^*_{2k+4}$ with two spokes missing.

Let $F^{**}_m$ be $F^*_m$ with any two spokes missing. Since $F^{**}_m$ is a subgraph of $W_m$, our computation verifies that $\n(F^{**}_m \cup K_1) = 0$ (i.e. $\n(F^{**}_m) = 1$) for all even $m \lqs 14$. This shows that $\n(W_n)=1$ for all odd $n$ up to $n=23$ (requiring $F^{**}_{12}$ and $F^{**}_{14}$) and $n=25$ (requiring $F^{**}_{14}$).
\end{proof}

From the proof of Theorem \ref{wheel} we see that the following conjecture implies Conjecture \ref{conjw}, i.e. that $\n(W_n)=1$ for all $n$.

\begin{conj}{\rm (Even fans with handles and  two spokes removed.)}
Let the spokes of $F^*_n$ be all edges connected to the hub except the edge connected to the degree $1$ vertex. For all  even $n\gqs 4$ the nim-value of  $F^*_n$ with any two spokes removed is $1$.
\end{conj}

 Exploring the nim-values of the move options for the fans $F_n$ and the fans with handles $F^*_n$ reveals the following patterns for the given $n$ values up to $14$ and we conjecture they hold for all $n$.

\SpecialCoor
\psset{griddots=5,subgriddiv=0,gridlabels=0pt}
\psset{xunit=0.5cm, yunit=0.5cm, runit=0.5cm}
\psset{linewidth=1pt}
\psset{dotsize=5pt 0,dotstyle=*}

\begin{figure}[ht]
\centering
\begin{pspicture}(1,4)(28,10.5) 

\psset{arrowscale=2,arrowinset=0.5}


\savedata{\mydataww}[
{{5,5},{8., 5.}, {7.59808, 6.5}, {5,5},{6.5, 7.59808}, {5., 8.}, {5,5},{3.5,
  7.59808}, {2.40192, 6.5}, {5,5},{2., 5.}}
]
\dataplot[plotstyle=line]{\mydataww}
\psline[linecolor=white, linewidth=2pt](5,5)(5,8)
\psline[linecolor=orange](5,5)(5,8)

\savedata{\mydata}[
{{8., 5.}, {7.59808, 6.5}, {6.5, 7.59808}, {5., 8.}, {3.5,
  7.59808}, {2.40192, 6.5}, {2., 5.}, {5,5}}
]
\dataplot[plotstyle=line]{\mydata}
\dataplot[plotstyle=dots]{\mydata}


\savedata{\mydataza}[
{{17,5}, {21., 5.}, {20.6955, 6.53073},{17,5},  {19.8284, 7.82843}, {18.5307,
  8.69552},{17,5},  {17., 9.}, {15.4693, 8.69552},{17,5},  {14.1716,
  7.82843}, {13.3045, 6.53073}, {17,5}, {13., 5.}}
]
\dataplot[plotstyle=line]{\mydataza}
\psline[linecolor=white, linewidth=2pt](17,5)(17,9)
\psline[linecolor=orange](17,5)(17,9)

\savedata{\mydataz}[
{{17,5}, {21., 5.}, {20.6955, 6.53073}, {19.8284, 7.82843}, {18.5307,
  8.69552}, {17., 9.}, {15.4693, 8.69552}, {14.1716,
  7.82843}, {13.3045, 6.53073}, {13., 5.}}
]
\dataplot[plotstyle=line]{\mydataz}
\dataplot[plotstyle=dots]{\mydataz}

\uput{3.5}[30](5,5){$1$}
\uput{3.5}[60](5,5){$1$}
\uput{3.5}[90](5,5){$1$}
\uput{3.5}[120](5,5){$1$}
\uput{3.5}[150](5,5){$1$}
\uput{3.5}[0](5,5){$3$}
\uput{3.5}[180](5,5){$3$}
\uput{0.5}[270](5,5){$1$}


\psline(25,8.5)(28,8.5)
\rput(26.5,9){$4$}
\psline[linecolor=orange](25,7)(28,7)
\rput(26.5,7.5){$0$}
\rput(26.5,5.5){$\f=2$}
\rput(26.5,4.5){$\n=2$}

\uput{4.5}[22.5](17,5){$1$}
\uput{4.5}[45](17,5){$1$}
\uput{4.5}[67.5](17,5){$1$}
\uput{4.5}[90](17,5){$1$}
\uput{4.5}[112.5](17,5){$1$}
\uput{4.5}[135](17,5){$1$}
\uput{4.5}[157.5](17,5){$1$}
\uput{4.5}[180](17,5){$3$}
\uput{4.5}[0](17,5){$3$}
\uput{0.5}[270](17,5){$1$}

\end{pspicture}
\caption{Nim-values for the move options of $F_7$ and $F_9$}
\label{odd-hd}
\end{figure}
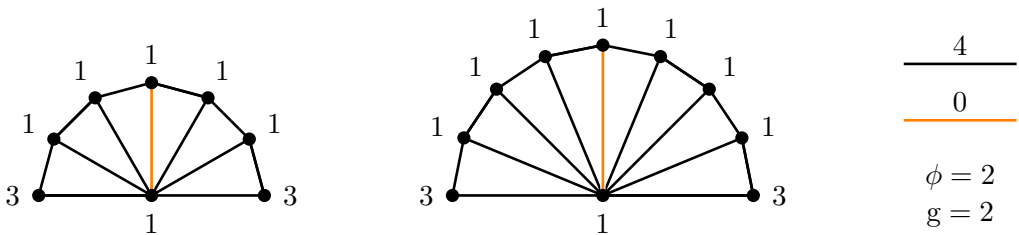

\SpecialCoor
\psset{griddots=5,subgriddiv=0,gridlabels=0pt}
\psset{xunit=0.5cm, yunit=0.5cm, runit=0.5cm}
\psset{linewidth=1pt}
\psset{dotsize=5pt 0,dotstyle=*}

\begin{figure}[ht]
\centering
\begin{pspicture}(1,4)(28,10.5) 

\psset{arrowscale=2,arrowinset=0.5}


\savedata{\mydataww}[
{{5, 5.},  {5.92705, 7.85317},{5, 5.}, {4.07295,
  7.85317}}
]
\dataplot[plotstyle=line]{\mydataww}

\psset{linecolor=orange}
\savedata{\mydataww}[
{{8., 5.}, {5, 5.}, {7.42705, 6.76336}, {5, 5.}, {2.57295, 6.76336}, {5, 5.},{2., 5.}}
]
\dataplot[plotstyle=line]{\mydataww}
\psset{linecolor=black}

\savedata{\mydata}[
{{8., 5.}, {7.42705, 6.76336}, {5.92705, 7.85317}, {4.07295,
  7.85317}, {2.57295, 6.76336}, {2., 5.}}
]
\dataplot[plotstyle=line]{\mydata}
\dataplot[plotstyle=dots]{\mydata}
\psdot(5,5)


\savedata{\mydataww}[
{{17.8901,
  8.89971}, {17,5}, {16.1099, 8.89971}}
]
\dataplot[plotstyle=line]{\mydataww}

\psset{linecolor=orange}
\savedata{\mydataww}[
{{21., 5.}, {17,5}, {20.6039, 6.73553}, {17,5}, {19.494, 8.12733}, {17,5}, {14.506, 8.12733}, {17,5}, {13.3961,
  6.73553},  {17,5},{13., 5.}}
]
\dataplot[plotstyle=line]{\mydataww}
\psset{linecolor=black}

\savedata{\mydata}[
{{21., 5.}, {20.6039, 6.73553}, {19.494, 8.12733}, {17.8901,
  8.89971}, {16.1099, 8.89971}, {14.506, 8.12733}, {13.3961,
  6.73553}, {13., 5.}}
]
\dataplot[plotstyle=line]{\mydata}
\dataplot[plotstyle=dots]{\mydata}
\psdot(17,5)

\uput{3.5}[36](5,5){$4$}
\uput{3.5}[72](5,5){$4$}
\uput{3.5}[108](5,5){$4$}
\uput{3.5}[144](5,5){$4$}

\uput{3.5}[0](5,5){$2$}
\uput{3.5}[180](5,5){$2$}
\uput{0.5}[270](5,5){$2$}

\psline(25,8.5)(28,8.5)
\rput(26.5,9){$1$}
\psline[linecolor=orange](25,7)(28,7)
\rput(26.5,7.5){$0$}
\rput(26.5,5.5){$\f=3$}
\rput(26.5,4.5){$\n=3$}

\uput{4.5}[25.7](17,5){$4$}
\uput{4.5}[51.4](17,5){$4$}
\uput{4.5}[77.1](17,5){$4$}
\uput{4.5}[102.9](17,5){$4$}
\uput{4.5}[128.6](17,5){$4$}
\uput{4.5}[154.3](17,5){$4$}

\uput{4.5}[180](17,5){$2$}
\uput{4.5}[0](17,5){$2$}
\uput{0.5}[270](17,5){$2$}

\end{pspicture}
\caption{Nim-values for the move options of $F_6$ and $F_8$}
\label{even-hd}
\end{figure}
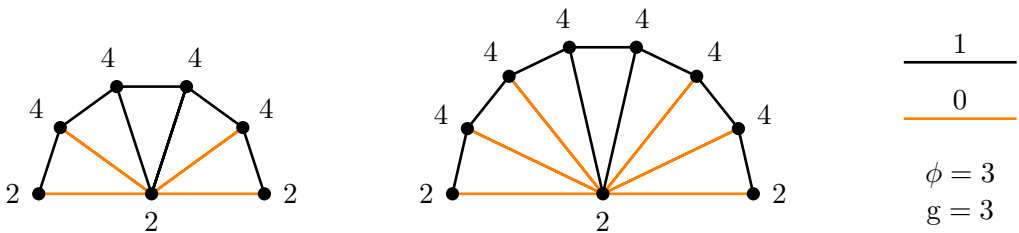

\begin{conj}{\rm (Fan options.)}
\begin{enumerate}
  \item For all  odd $n\gqs 7$ the nim-values of the options in the fan $F_n$ are as follows. All vertices have value $1$ except the degree $2$ vertices which have value $3$. All edges have value $4$ except the edge on the axis of symmetry with value $0$. See the examples in Figure  \ref{odd-hd}.
  \item For all  even $n\gqs 4$ the nim-values of the options in the fan $F_n$ are as follows. All vertices have value $4$ except the degree $2$ vertices and the central vertex which have value $2$. All edges on the rim along with the two spokes on each side of the axis of symmetry have value $1$. The remaining spokes have value $0$. Hence $\n(F_n)=3$ for $n$ even. See the examples in Figure  \ref{even-hd}.
\end{enumerate}
\end{conj}

\SpecialCoor
\psset{griddots=5,subgriddiv=0,gridlabels=0pt}
\psset{xunit=0.5cm, yunit=0.5cm, runit=0.5cm}
\psset{linewidth=1pt}
\psset{dotsize=5pt 0,dotstyle=*}

\begin{figure}[ht]
\centering
\begin{pspicture}(1,2.5)(28,10.5) 

\psset{arrowscale=2,arrowinset=0.5}


\savedata{\mydataww}[
{{5,3},{5,5},{8., 5.}, {7.59808, 6.5}, {5,5},{6.5, 7.59808}, {5., 8.}, {5,5},{3.5,
  7.59808}, {2.40192, 6.5}, {5,5},{2., 5.}}
]
\dataplot[plotstyle=line]{\mydataww}

\savedata{\mydata}[
{{5,3},{5,5},{8., 5.}, {7.59808, 6.5}, {6.5, 7.59808},  {5., 8.}, {3.5,
  7.59808}, {2.40192, 6.5}, {2., 5.}, {5,5}}
]
\dataplot[plotstyle=line]{\mydata}
\psline[linecolor=white, linewidth=2pt](7.59808, 6.5)(6.5, 7.59808)
\psline[linecolor=orange](7.59808, 6.5)(6.5, 7.59808)
\psline[linecolor=white, linewidth=2pt](3.5,7.59808)(2.40192, 6.5)
\psline[linecolor=orange](3.5,7.59808)(2.40192, 6.5)
\dataplot[plotstyle=dots]{\mydata}


\savedata{\mydataza}[
{{17,3}, {17,5}, {21., 5.}, {20.6955, 6.53073},{17,5},  {19.8284, 7.82843}, {18.5307,
  8.69552},{17,5},  {17., 9.}, {15.4693, 8.69552},{17,5},  {14.1716,
  7.82843}, {13.3045, 6.53073}, {17,5}, {13., 5.}}
]
\dataplot[plotstyle=line]{\mydataza}

\savedata{\mydataz}[
{{17,3}, {17,5}, {21., 5.}, {20.6955, 6.53073}, {19.8284, 7.82843}, {18.5307,
  8.69552}, {17., 9.}, {15.4693, 8.69552}, {14.1716,
  7.82843}, {13.3045, 6.53073}, {13., 5.}}
]
\dataplot[plotstyle=line]{\mydataz}
\psline[linecolor=white, linewidth=2pt](20.6955, 6.53073)(19.8284, 7.82843)
\psline[linecolor=orange](20.6955, 6.53073)(19.8284, 7.82843)
\psline[linecolor=white, linewidth=2pt](14.1716,7.82843)(13.3045, 6.53073)
\psline[linecolor=orange](14.1716,7.82843)(13.3045, 6.53073)
\dataplot[plotstyle=dots]{\mydataz}

\uput{3.5}[30](5,5){$2$}
\uput{3.5}[60](5,5){$2$}
\uput{3.5}[90](5,5){$2$}
\uput{3.5}[120](5,5){$2$}
\uput{3.5}[150](5,5){$2$}
\uput{3.5}[0](5,5){$4$}
\uput{3.5}[180](5,5){$4$}
\uput{0.5}[315](5,5){$0$}
\uput{0.5}[355](5,3){$2$}

\psline(25,8.5)(28,8.5)
\rput(26.5,9){$3$}
\psline[linecolor=orange](25,7)(28,7)
\rput(26.5,7.5){$0$}
\rput(26.5,5.5){$\f=1$}
\rput(26.5,4.5){$\n=1$}

\uput{4.5}[22.5](17,5){$2$}
\uput{4.5}[45](17,5){$2$}
\uput{4.5}[67.5](17,5){$2$}
\uput{4.5}[90](17,5){$2$}
\uput{4.5}[112.5](17,5){$2$}
\uput{4.5}[135](17,5){$2$}
\uput{4.5}[157.5](17,5){$2$}
\uput{4.5}[180](17,5){$4$}
\uput{4.5}[0](17,5){$4$}
\uput{0.5}[315](17,5){$0$}
\uput{0.5}[355](17,3){$2$}

\end{pspicture}
\caption{Nim-values for the move options of $F^*_8$ and $F^*_{10}$}
\label{even-fan}
\end{figure}
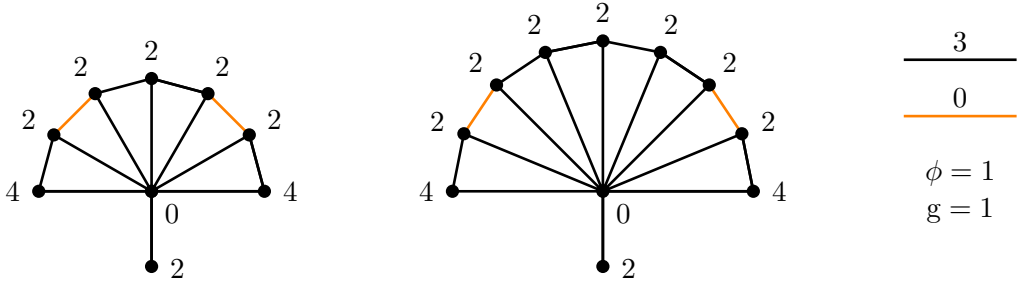

\SpecialCoor
\psset{griddots=5,subgriddiv=0,gridlabels=0pt}
\psset{xunit=0.5cm, yunit=0.5cm, runit=0.5cm}
\psset{linewidth=1pt}
\psset{dotsize=5pt 0,dotstyle=*}

\begin{figure}[ht]
\centering
\begin{pspicture}(1,2.5)(28,10.5) 

\psset{arrowscale=2,arrowinset=0.5}


\savedata{\mydataww}[
{{5,5},{8., 5.}, {7.70291, 6.30165},{5,5}, {6.87047, 7.34549}, {5.66756,
  7.92478},{5,5}, {4.33244, 7.92478}, {3.12953, 7.34549},{5,5}, {2.29709,
  6.30165}, {2., 5.},{5,5}}
]
\dataplot[plotstyle=line]{\mydataww}

\savedata{\mydata}[
{{5,3},{5,5},{8., 5.}, {7.70291, 6.30165}, {6.87047, 7.34549}, {5.66756,
  7.92478}, {4.33244, 7.92478}, {3.12953, 7.34549}, {2.29709,
  6.30165}, {2., 5.}}
]
\dataplot[plotstyle=line]{\mydata}
\dataplot[plotstyle=dots]{\mydata}

\uput{3.5}[25.7](5,5){$3$}
\uput{3.5}[51.4](5,5){$0$}
\uput{3.5}[77.1](5,5){$3$}
\uput{3.5}[102.9](5,5){$3$}
\uput{3.5}[128.6](5,5){$0$}
\uput{3.5}[154.3](5,5){$3$}

\uput{3.5}[0](5,5){$1$}
\uput{3.5}[180](5,5){$1$}
\uput{0.5}[315](5,5){$3$}
\uput{0.5}[355](5,3){$3$}


\savedata{\mydataww}[
{{17,5}, {21., 5.}, {20.7588, 6.36808},{17,5},  {20.0642, 7.57115}, {19.,
  8.4641},{17,5},  {17.6946, 8.93923}, {16.3054, 8.93923}, {17,5}, {15.,
  8.4641}, {13.9358, 7.57115},{17,5},  {13.2412, 6.36808}, {13., 5.},{17,5}}
]
\dataplot[plotstyle=line]{\mydataww}

\savedata{\mydata}[
{{17,3},{17,5},{21., 5.}, {20.7588, 6.36808}, {20.0642, 7.57115}, {19.,
  8.4641}, {17.6946, 8.93923}, {16.3054, 8.93923}, {15.,
  8.4641}, {13.9358, 7.57115}, {13.2412, 6.36808}, {13., 5.}}
]
\dataplot[plotstyle=line]{\mydata}
\dataplot[plotstyle=dots]{\mydata}

\psline(25,7)(28,7)
\rput(26.5,7.5){$2$}

\rput(26.5,5.5){$\f=0$}
\rput(26.5,4.5){$\n=4$}

\uput{4.5}[20](17,5){$3$}
\uput{4.5}[40](17,5){$0$}
\uput{4.5}[60](17,5){$3$}
\uput{4.5}[80](17,5){$3$}
\uput{4.5}[100](17,5){$3$}
\uput{4.5}[120](17,5){$3$}
\uput{4.5}[140](17,5){$0$}
\uput{4.5}[160](17,5){$3$}

\uput{4.5}[180](17,5){$1$}
\uput{4.5}[0](17,5){$1$}
\uput{0.5}[315](17,5){$3$}
\uput{0.5}[355](17,3){$3$}

\end{pspicture}
\caption{Nim-values for the move options of $F^*_9$ and $F^*_{11}$}
\label{odd-fan}
\end{figure}
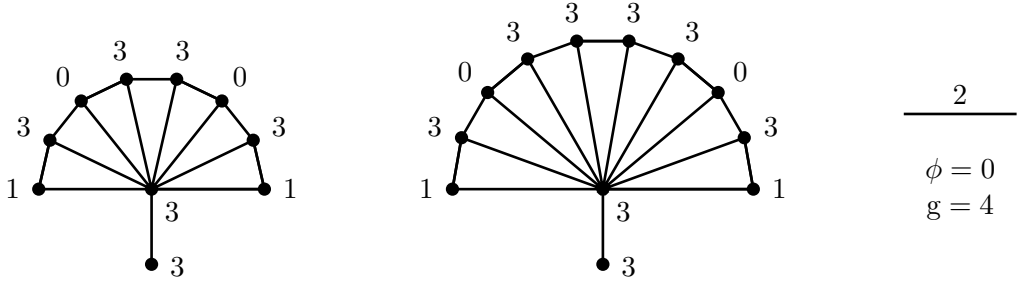

\begin{conj}{\rm (Fan with a handle options.)}
\begin{enumerate}
  \item For all  even $n\gqs 6$ the nim-values of the options in the fan $F^*_n$ are as follows. All vertices have value $2$ except the degree $2$ vertices which have value $4$ and the central vertex with value $0$. All edges have value $3$ except two on the rim a distance $1$ from the degree $2$ vertices. They have  value $0$. See the examples in Figure  \ref{even-fan}.
  \item For all  odd $n\gqs 7$ the nim-values of the options in the fan $F^*_n$ are as follows. Almost all vertices have value $3$. The exceptions are the degree $2$ vertices with value $1$ and the two vertices on the rim a distance $2$ from these with value $0$. All edges  have value $2$. Hence $\n(F^*_n)=4$ for $n$ odd. See the examples in Figure  \ref{odd-fan}.
\end{enumerate}
\end{conj}

It is clear by the symmetry lemma that $\n(F_n)=2$ for $n$ odd and that deleting the edge on the axis of symmetry gives nim-value $0$. Similarly we may show that $\n(F^*_n)=1$ for $n$ even and the highlighted edges in Figure \ref{even-fan} give nim-value $0$ when removed because the graph simplifies to a cycle.
It is remarkable that, except for those just mentioned, all other edges in each $F^*_n$ seem to have the same nim-value: $3$ for $n$ even and $2$ for $n$ odd. Perhaps proving the patterns in these conjectures requires characterizing when $\n(H)=\f(H)$ for subgraphs of $W_n$, similarly to Theorem \ref{AB2}.

Computer calculations indicate that a subgraph $H$ of $W_n$ with $\f(H)=2$ never has  nim-value $0$. This leads us to the following conjecture.

\begin{conj} \label{edcon}
Let $H$ be any subgraph of $W_n$ for $n \gqs 3$. Suppose $\f(H)=2$ (i.e. $H$ has an even number of vertices and an odd number of edges). Then there exists an edge of $H$ so that removing it gives a graph with nim-value $0$.
\end{conj}

This conjecture is true for all subgraphs with $\f=2$ of $W_n$  for $3\lqs n \lqs 14$. For example, the graphs shown in Figure \ref{odd-hd} have $\f=2$ and contain a single edge move with nim-value $0$. The conjecture is also true for bipartite graphs but not for general graphs and fails for instance for  the three graphs shown in Figure \ref{exep}.
\SpecialCoor
\psset{griddots=5,subgriddiv=0,gridlabels=0pt}
\psset{xunit=0.5cm, yunit=0.5cm, runit=0.5cm}
\psset{linewidth=1pt}
\psset{dotsize=5pt 0,dotstyle=*}
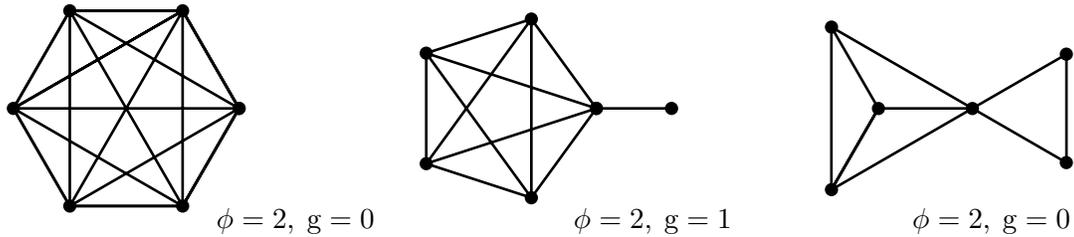
\begin{figure}[ht]
\centering
\begin{pspicture}(2,1.5)(30,8) 

\psset{arrowscale=2,arrowinset=0.5}


\savedata{\mydataww}[
{{8., 5.}, {6.5, 7.59808}, {3.5, 7.59808}, {8., 5.},  {2., 5.}, {3.5,
  2.40192}, {8., 5.}, {6.5, 2.40192}, {8., 5.}, {6.5, 7.59808}, {3.5, 7.59808}, {2., 5.}, {6.5, 7.59808}, {3.5,
  2.40192}, {6.5, 2.40192}, {6.5, 7.59808}, {2., 5.}, {6.5, 2.40192},  {3.5, 7.59808}, {3.5,
  2.40192} }
]
\dataplot[plotstyle=line]{\mydataww}

\savedata{\mydata}[
{{8., 5.}, {6.5, 7.59808}, {3.5, 7.59808}, {2., 5.}, {3.5,
  2.40192}, {6.5, 2.40192}, {8., 5.}}
]
\dataplot[plotstyle=line]{\mydata}
\dataplot[plotstyle=dots]{\mydata}


\savedata{\mydataww}[
{{17.5, 5.},  {12.9775, 6.46946},  {15.7725, 2.62236}, {15.7725, 7.37764},{12.9775,
  3.53054}, {17.5, 5.}}
]
\dataplot[plotstyle=line]{\mydataww}

\savedata{\mydata}[
{{17.5, 5.}, {15.7725, 7.37764}, {12.9775, 6.46946}, {12.9775,
  3.53054}, {15.7725, 2.62236}, {17.5, 5.}}
]
\dataplot[plotstyle=line]{\mydata}
\dataplot[plotstyle=dots]{\mydata}

\psline(17.5, 5.)(19.5, 5.)
\psdot(19.5,5)

\psline(27.5,5)(30,6.44338)(30,3.55662)(27.5,5)(25,5)(23.75, 2.83494)(23.75, 7.16506)(25,5)(23.75, 2.83494)(27.5,5)(23.75, 7.16506)
\psdots(25,5)(27.5,5)(23.75, 7.16506)(23.75, 2.83494)(30,6.44338)(30,3.55662)
\rput(9.5,2){$\f=2, \ \n=0$}
\rput(19,2){$\f=2, \ \n=1$}
\rput(28,2){$\f=2, \ \n=0$}

\end{pspicture}
\caption{Three examples with $\f=2$ but no winning edge moves}
\label{exep}
\end{figure}
These graphs have $\f=2$ but no edge moves give nim-value $0$ and a computer search shows they are the only graphs on $6$ or fewer vertices with this property. Interestingly, they are of the form $K_i \cup K_j$ for $i+j=7$ with a vertex of $K_i$ and  $K_j$ identified.

We list five straightforward consequences of Conjecture \ref{edcon} with $H$ any subgraph of a wheel:
\begin{enumerate}
  \item If $H$ has an odd number of edges then there exists an edge of $H$ so that removing it gives a nim-value of $\f(H)\oplus 2$.
  \item If $H$ has an even number of edges then $\n(H)=\f(H)$ or there exists an edge of $H$ so that removing it gives a nim-value of $\f(H)\oplus 2$ (or both).
  \item If $H$ has an even number of vertices then $\n(H)=0$ implies $\f(H)=0$.
  \item If $\f(H)=0$ then any winning move must remove a vertex.
  \item Lastly we note that Conjecture \ref{edcon} implies  Conjecture \ref{conjw}, i.e. that $\n(W_n)$ always equals $1$. To see this implication, recall from the proof of Theorem \ref{wheel} that $\n(W_n)=1$ for $n$ odd follows if we can show that there is a winning response to $Q_n$ with a spoke removed. Since $Q_n$ with a spoke removed has $\f=2$, Conjecture \ref{edcon} implies there is a winning edge response.
\end{enumerate}


{\small
\bibliography{gamedata}
}

{\small 
\vskip 5mm
\noindent
\textsc{Dept. of Math, The CUNY Graduate Center, 365 Fifth Avenue, New York, NY 10016-4309, U.S.A.}

\noindent
{\em E-mail address:} \texttt{cosullivan@gc.cuny.edu}
}

\end{document}